\documentclass[final,hidelinks]{siamart1116}

\usepackage{amsopn}
\usepackage{amsfonts}
\usepackage{amssymb}
\usepackage{graphicx}
\usepackage{epstopdf}
\usepackage{float}
\usepackage{mathrsfs}
\usepackage{mathtools}
\usepackage[inline]{enumitem}
\usepackage{bm}
\usepackage[only,llbracket,rrbracket]{stmaryrd}
\SetSymbolFont{stmry}{bold}{U}{stmry}{m}{n}
\usepackage[shortcuts]{extdash}
\usepackage[caption=false]{subfig}
\usepackage{tikz}
\usepackage{multirow}
\usepackage{makecell}
\usepackage{diagbox}

\newcommand\where{; \allowbreak \nonscript\; \mathopen{}}

\DeclarePairedDelimiterX\set[1]{\lbrace}{\rbrace}{\nonscript\,  #1 \nonscript\,}

\newcommand\restr[2]{{\left.\kern-\nulldelimiterspace #1 \right\rvert_{#2}}}

\newcommand{\ol}[1]{\overline{#1}}

\newcommand{\pardf}[2]{{\partial #1 / \partial #2}}

\newcommand{\lp}{\left(}
\newcommand{\rp}{\right)}
\newcommand{\lb}{\left[}
\newcommand{\rb}{\right]}
\newcommand{\ldb}{\llbracket}
\newcommand{\rdb}{\rrbracket}

\newcommand{\li}{\langle}
\newcommand{\ri}{\rangle}

\newcommand{\lv}{\lvert}
\newcommand{\rv}{\rvert}

\newcommand{\lV}{\lVert}
\newcommand{\rV}{\rVert}

\newcommand\bff{\bm{f}}

\newcommand\bn{\bm{n}}

\newcommand\br{\bm{r}}
\newcommand\bs{\bm{s}}

\newcommand\bv{\bm{v}}

\newcommand\bx{\bm{x}}

\newcommand\bzero{\bm{0}}

\newcommand\btau{\bm{\tau}}

\newcommand\bbR{\mathbb{R}}

\newcommand\cC{\mathcal{C}}

\newcommand\cF{\mathcal{F}}
\newcommand\cG{\mathcal{G}}

\newcommand\cN{\mathcal{N}}
\newcommand\cO{\mathcal{O}}

\newcommand\cQ{\mathcal{Q}}
\newcommand\cR{\mathcal{R}}
\newcommand\cS{\mathcal{S}}

\newcommand\cU{\mathcal{U}}
\newcommand\cV{\mathcal{V}}

\newcommand\bcR{\bm{\cR}}
\newcommand\bcS{\bm{\cS}}

\newcommand\hcF{\hat{\cF}}

\newcommand\sC{\mathscr{C}}

\newcommand\hp{\hat{p}}
\newcommand\hq{\hat{q}}
\newcommand\hu{\hat{u}}
\newcommand\hv{\hat{v}}

\newcommand\hmu{\hat{\mu}}

\newcommand\hpsi{\hat{\psi}}

\newcommand\tq{\tilde{q}}

\newcommand\tu{\tilde{u}}

\newcommand\tpsi{\tilde{\psi}}

\newcommand\half{\frac{1}{2}}
\newcommand\halff{{1/2}}

\newcommand\mhalff{{-1/2}}
\newcommand\mone{{-1}}

\newcommand{\nn}{\nonumber}
\newcommand\col{\colon}

\newcommand\dd{\mathop{}\!\mathrm{d}}
\newcommand\lap{\Delta}
\newcommand\grad{\nabla}
\newcommand\gradp{\grad^\perp}
\newcommand\gradd{\grad\cdot}

\let\div\relax\DeclareMathOperator{\div}{div}
\DeclareMathOperator{\curl}{curl}

\DeclareMathOperator{\myspan}{span}

\DeclareMathOperator*{\argmin}{argmin}
\DeclareMathOperator*{\loc}{loc}

\allowdisplaybreaks[4]

\newcommand{\Hdiv}[1]{H(\div\where #1)}
\newcommand{\Hone}[2]{H^{1}_{0,#1}(#2)}
\newcommand{\Hmone}[2]{H^{-1}_{#1}(#2)}
\newcommand\wgradd{\grad_{\mathfrak{d}}\cdot}
\newcommand\cVh[1]{\cV^{\smallh}_{#1}}
\newcommand\scVh[1]{\cV^{\smallh}_{#1}}
\newcommand\smallDelta{\scalebox{0.5}{$\Delta$}}
\newcommand\smallh{h}

\newcommand{\TheTitle}{Least-Squares for Hyperbolic Balance Laws} 
\newcommand{\TheAuthors}{D. Z. Kalchev, T. A. Manteuffel}

\headers{\TheTitle}{\TheAuthors}

\title{{A Least-Squares Finite Element Method Based on the Helmholtz Decomposition for Hyperbolic Balance Laws}\thanks{This work was performed under the auspices of the U.S. Department of Energy under grant numbers (SC) DE-FC02-03ER25574, (NNSA) DE-NA0002376, and contract (NNSA) DE-AC52-07NA27344 (LLNL-JRNL-756408); Lawrence Livermore National Laboratory (LLNL) under contract B614452.}}

\author{
  Delyan Z. Kalchev\footnotemark[2]\ \footnotemark[3]
  \and
  Thomas A. Manteuffel\footnotemark[2]
}

\ifpdf
  \DeclareGraphicsExtensions{.eps,.pdf,.png,.jpg}
\else
  \DeclareGraphicsExtensions{.eps}
\fi

\ifpdf
\hypersetup{
  pdftitle={\TheTitle},
  pdfauthor={\TheAuthors}
}
\fi

\numberwithin{theorem}{section}
\numberwithin{equation}{section}
\newsiamremark{remark}{Remark}

\newcommand{\blue}[1]{#1}

\begin{document}

\maketitle
\renewcommand{\thefootnote}{\fnsymbol{footnote}}
\footnotetext[2]{Department of Applied Mathematics, 526 UCB, University of Colorado at Boulder, Boulder, CO 80309-0526, USA (\email{delyan.kalchev@colorado.edu}, \email{tmanteuf@colorado.edu}).}
\footnotetext[3]{Center for Applied Scientific Computing, Lawrence Livermore National Laboratory, P.O. Box 808, L-561, Livermore, CA 94551, USA (\email{kalchev1@llnl.gov}).}

\begin{abstract}
In this paper, a least-squares finite element method for scalar nonlinear hyperbolic balance laws is proposed and studied. The approach is based on a formulation that utilizes an appropriate Helmholtz decomposition of the flux vector and is related to the standard notion of a weak solution. This relationship, together with a corresponding connection to negative-norm least-squares, is described in detail. As a consequence, an important numerical conservation theorem is obtained, similar to the famous Lax-Wendroff theorem. The numerical conservation properties of the method in this paper do not fall precisely in the framework introduced by Lax and Wendroff, but they are similar in spirit as they guarantee that when $L^2$ convergence holds, the resulting approximations approach a weak solution to the hyperbolic problem. The least-squares functional is continuous and coercive in an $H^{-1}$-type norm, but not $L^2$-coercive. Nevertheless, the $L^2$ convergence properties of the method are discussed. Convergence can be obtained either by an explicit regularization of the functional, that provides control of the $L^2$ norm, or by properly choosing the finite element spaces, providing implicit control of the $L^2$ norm. Numerical results for the inviscid Burgers equation with discontinuous source terms are shown, demonstrating the $L^2$ convergence of the obtained approximations to the physically admissible solution. The numerical method utilizes a least-squares functional, minimized on finite element spaces, and a Gauss-Newton technique with nested iteration. We believe that the linear systems encountered with this formulation are amenable to multigrid techniques and combining the method with adaptive mesh refinement would make this approach an efficient tool for solving balance laws (this is the focus of a future study).
\end{abstract}


\section{Introduction}

Hyperbolic conservation and balance laws arise often in practice, especially in problems of fluid mechanics \cite{LeVequeHCL,LeVequeHyperbolic,GodlewskiHCL,2012BalanceLaws}. A variety of numerical schemes have been developed for the solution of such problems. This includes finite difference and finite volume \cite{LeVequeHCL,LeVequeHyperbolic,GodlewskiHCL,2012BalanceLaws}, as well as finite element methods. In the field of finite elements, notably, discontinuous Galerkin (DG) methods (see \cite{DiPietroFEM} and the references therein) are often utilized for the solution of hyperbolic partial differential equations (PDEs) as well as SUPG (streamline-upwind/Petrov-Galerkin) methods \cite{1993SUPGCompress,2001ComparativeLS}. Furthermore, Guermond, Popov et al. proposed and studied the so-called entropy viscosity method (see \cite{2011EntropyViscNonlin,2013ImplEntropyVisc,2014MaxPreservingCons}) and Dobrev, Kolev, Rieben et al. developed and implemented approaches for the Euler equation of compressible hydrodynamics in a moving Lagrangian frame (cf., \cite{2011CurvilinearLagrangian,2012HOCurvilinearLagrangian}; see also \cite{2017DGtransport}). This paper focuses on least-squares finite element techniques. Least-squares methods \cite{BochevLSFEM} have been developed for a variety of problems, including linear \cite{2004LinearHyperbolic,2001ErrorLS,2001ComparativeLS,2018LLstarAndInverse} and nonlinear \cite{2005HdivHyperbolic} first-order hyperbolic PDEs; see also \cite{OlsonPhDthesis,2005LSShallowWater,1988LSHypSystems}. Those approaches utilize appropriate least-squares minimization principles to obtain finite element discretizations of PDEs. Computationally, the problem is reduced to solving linear algebraic systems with symmetric positive definite matrices associated with quadratic minimization problems. Least-squares functionals provide natural a posteriori error estimates for adaptive local refinement (ALR). Particularly, ALR in the setting of nested iteration with Gauss-Newton techniques and in concomitance with algebraic multigrid (AMG) linear solvers can provide efficient methods for solving nonlinear equations to a certain error tolerance, while optimizing the respective computational cost for obtaining the desired level of accuracy; see \cite{1997AMR,2008ACE,2011ACE}. These useful aspects of least-squares techniques, their general flexibility, and the good shock-capturing capabilities, demonstrated in \cite{2005HdivHyperbolic} and this paper, prompt the present work on developing a least-squares discretization method for hyperbolic balance laws. In order to be used as a local error estimate in ALR (the focus of a future study), the least-squares functional proposed here requires an additional standard Poisson solve, while this solve is not needed for obtaining the final linear algebraic systems.

The notion of a weak solution \cite{1954WeakSols} is of fundamental importance for hyperbolic conservation and balance laws, since it allows the consideration of non-smooth solutions, which are of practical interest. In the numerical treatment of these problems, a related important property is that the obtained approximations, if they converge, approach a weak solution \cite{1960Conservation} of the respective hyperbolic PDE. Such a property is associated with the ability of the numerical method to correctly approximate weak solutions (i.e., solutions with discontinuities) to nonlinear problems \cite{1994NoncosSchemes,LeVequeHCL,LeVequeHyperbolic,GodlewskiHCL}. This is linked to the famous Lax-Wendroff theorem established in \cite{1960Conservation}. Based on that result, it has become standard, especially in the context of finite volume \cite{LeVequeHyperbolic,GodlewskiHCL} and DG finite element \cite{DiPietroFEM} methods, to consider so-called conservative schemes that posses a discrete conservation property. Such a conservation property in the Lax-Wendroff theorem provides a sufficient condition for approximating weak solutions to nonlinear hyperbolic PDEs. As demonstrated in \cite{2005HdivHyperbolic}, the discrete conservation property, while sufficient, is not necessary for obtaining convergence to a weak solution -- a fact that is also utilized in this paper. As in \cite{2005HdivHyperbolic}, the considerations here do not precisely abide by the framework provided by Lax and Wendroff in \cite{1960Conservation}. Namely, based on the utilization of an appropriate least-squares minimization principle, related to the notion of a weak solution, we establish the important and desired numerical conservation property that approximations obtained by the method of this paper approach a weak solution to the hyperbolic PDE of interest. This, together with the numerically demonstrated, in \cref{sec:numerical}, $L^2$ convergence to the physically admissible solution and good shock-capturing capabilities, motivates the derivation and consideration of the particular formulation in this paper.

The main contributions of this paper are summarized as follows. This work proposes and studies a general least-squares finite element formulation for scalar nonlinear hyperbolic balance laws, which is associated with the standard notion of a weak solution. The fundamental idea is related to and extends the considerations in \cite{2005HdivHyperbolic}. The method of this paper utilizes an appropriate Helmholtz decomposition of the flux function to obtain a least-squares functional. The particular Helmholtz decomposition, unlike \cite{2005HdivHyperbolic}, allows the accommodation of source terms and also a natural treatment of the inflow boundary conditions; see the end of \cref{ssec:helmholtz}. Any weak solution to the hyperbolic PDE yields a global minimum value\footnote{This value is implicitly obtainable from the properties of the functional, while it requires an additional Poisson solve to be explicitly evaluated.} of zero for the functional. If the Helmholtz decomposition is exact, the functional is equal to an $H^{-1}$ norm of the residual of the equation, which is, in turn, associated with a weak form of the PDE. In this case, any zero of the functional is a weak solution. If the Helmholtz decomposition is not exact, the functional remains coercive with respect to the $H^{-1}$ norm, but continuity holds up to the accuracy of the decomposition (\cref{thm:relation}). Minimization of the proposed functional, which is mesh-dependent and based on the appropriate Helmholtz decomposition, over a nested sequence of finite element spaces yields a sequence of approximations to the exact solution. Standard approximation properties yield convergence of the functional to its minimum value of zero with rates determined by the smoothness of the approximated functions in accordance with standard finite element theory (\cref{cor:zerominconv}). Under mild hypotheses on the flux function, if the sequence of approximations converges boundedly in the $L^2$ norm, the limit is a weak solution, showing that the method is conservative (\cref{thm:weakconserv}). The least-squares functional is continuous in $L^2$, but not coercive. Thus, in general, the convergence in $L^2$ is not guaranteed by means of standard finite element theory. A novel and unorthodox approach for the analysis of the convergence is developed. Convergence in $L^2$ of the method is proved (\cref{thm:basicconv}) under additional and reasonable assumptions. Namely, if the finite element spaces satisfy a simple inf-sup condition, which provides an implicit mesh-dependent $L^2$ coercivity with a bound $c h^\alpha$, and a basic approximation estimate holds with a bound $Ch^\varkappa$, then the relation $\varkappa > \alpha$ guarantees the $L^2$ convergence of the method with a rate $\cO(h^{\varkappa - \alpha})$. Furthermore, a ``restricted'' version of the inf-sup condition is suggested, that improves on the convergence result (\cref{thm:moreconv}). The implicit mesh-dependent $L^2$ coercivity was discussed in \cite{PhDthesis}. Also, a spacial functional regularization (\cref{ssec:reg}) is proposed that explicitly provides (mesh-dependent) $L^2$ coercivity (\cref{thm:regcoerc}). Numerical results (\cref{sec:numerical}) in this paper demonstrate that the $L^2$ convergence to the physically admissible limit (a theoretical proof of this is the focus of a future study) is obtained for both the regularized and non-regularized formulations.

The ideas in the paper, for clarity and simplicity, are presented in a two\-/dimensional setting (e.g., one spatial and one time, or two spatial dimensions). The considerations can be extended for a three- or four-dimensional setting by utilizing a corresponding Helmholtz decomposition. In particular, this may require four-dimensional finite elements, which is an active field of research; cf., \cite{4D} and the references therein.

The outline of the rest of the paper is the following. Basic notions and the utilized Helmholtz decomposition are presented in \cref{sec:basic}. \Cref{sec:balance} contains a basic overview of scalar hyperbolic balance laws. In \cref{sec:lsapproach}, the least-squares formulations of interest are introduced, and they are analyzed and studied in more detail in \cref{sec:analysis}, including numerical conservation (\cref{ssec:conservation}, \cref{thm:weakconserv}), the special regularization (\cref{ssec:reg}, \cref{thm:regcoerc}), and convergence properties (\cref{ssec:conv}, \cref{thm:basicconv,thm:moreconv}). \Cref{sec:numerical} is devoted to numerical results, employing nested iteration with uniform mesh refinement. The conclusions and future work are in the final \cref{sec:conclusions}.

\section{Basic definitions and the Helmholtz decomposition}
\label{sec:basic}

Here, basic notation and definitions are presented and the Helmholtz decomposition, that is studied in this paper, is stated.

Let $\Omega$ be an open, bounded, and simply connected subset of $\bbR^2$ with a Lipschitz-continuous boundary, $\Gamma = \partial \Omega$. In the context of time-dependent hyperbolic problems, $\bbR^2$ represents the space-time, i.e., it is the $tx$-space, where $t$ and $x$ are the independent variables. Accordingly, $\gradd$ denotes the space-time divergence, i.e., $\gradd \bv = \partial_t v_1 + \partial_x v_2$, for any appropriate vector field $\bv\col \Omega \to \bbR^2$, $\bv = [v_1, v_2]$. Similarly, $\grad$ and $\gradp$ are space-time differential operators defined as $\grad v = [\partial_t v, \partial_x v]$ and $\gradp v = [\partial_x v, -\partial_t v]$, for any appropriate scalar function $v\col \Omega \to \bbR$. 

Let $\Gamma_S \subset \Gamma$ be a portion of the boundary of $\Omega$ with a nonzero curve measure. The space of $H^1(\Omega)$ functions with vanishing traces on $\Gamma_S$ is denoted as $\Hone{\Gamma_S}{\Omega} = \set{\phi \in H^1(\Omega) \where \phi = 0 \text{ on } \Gamma_S}$. It is convenient to consider the $H^1(\Omega)$ seminorm: $\lv \phi \rv_1 = \lV \grad \phi \rV$, for all $\phi \in H^1(\Omega)$, where $\lV \cdot \rV$ denotes the norms in both $L^2(\Omega)$ and $[L^2(\Omega)]^2$. Owing to Poincar{\'e}'s inequality, $\lv \cdot \rv_1$ is a norm in $\Hone{\Gamma_S}{\Omega}$, equivalent to the $H^1(\Omega)$ norm. Therefore, in this paper, $\Hone{\Gamma_S}{\Omega}$ is endowed with the norm $\lv \cdot \rv_1$ and, clearly, it is a Hilbert space with respect to that norm.

It is customary to define the dual of a positive-order Sobolev space as a ``negative-order'' Sobolev space. Following this practice, the dual space of $\Hone{\Gamma_S}{\Omega}$ is denoted by $\Hmone{\Gamma_S}{\Omega}$ and it is endowed with the respective functional norm $\lV \ell \rV_{-1,\Gamma_S} = \sup_{\phi \in \Hone{\Gamma_S}{\Omega}} \lv \ell(\phi) \rv/\lv \phi \rv_{1}$, for all $\ell \in \Hmone{\Gamma_S}{\Omega}$, where, to simplify notation, it is understood that $\phi \ne 0$ in the supremum. In particular, in the special case when $\Gamma_S \equiv \Gamma$, the commonly used notation is $H^1_0(\Omega) = \Hone{\Gamma}{\Omega}$ and $H^{-1}(\Omega) = \Hmone{\Gamma}{\Omega}$.

The inner products in both $L^2(\Omega)$ and $[L^2(\Omega)]^2$ are denoted by $(\cdot,\cdot)$. Following the notation in \cite{GiraultFEM}, the inner product in $L^2(\Gamma)$ is denoted by $\li \cdot,\cdot \ri_\Gamma$. By extending, as customary, the $L^2(\Gamma)$ inner product into a duality pairing, $\li \cdot,\cdot \ri_\Gamma$ is also used to denote the duality pairing between $H^\mhalff(\Gamma)$ and $H^\halff(\Gamma)$, where $H^\halff(\Gamma)$ is the space of traces on $\Gamma$ of functions in $H^1(\Omega)$ and $H^\mhalff(\Gamma)$ is its dual.

The Sobolev space of square integrable vector fields on $\Omega$ with square integrable divergence is denoted as $\Hdiv{\Omega} = \set{\bv \in [L^2(\Omega)]^2 \where \gradd \bv \in L^2(\Omega)}$, where $\gradd \bv \in L^2(\Omega)$, for $\bv \in [L^2(\Omega)]^2$, is understood in the sense that there exists a (unique) function $v \in L^2(\Omega)$ such that
\begin{equation}\label{eq:distdiv}
-(\bv, \grad \phi) = (v, \phi),\quad\forall \phi \in H^1_0(\Omega),
\end{equation}
in which case $\gradd \bv = v \in L^2(\Omega)$; see \cite{GiraultFEM}.

Using the notation above, we can write the following Green's formula \cite{GiraultFEM}:
\begin{equation}\label{eq:greendiv}
(\bv, \grad \phi) + (\gradd \bv, \phi) = \li \bv\cdot\bn, \phi \ri_\Gamma, \quad \forall \bv \in \Hdiv{\Omega},\; \forall \phi\in H^1(\Omega),
\end{equation}
where $\bn$ is the unit outward normal to $\Gamma$.

Finally, let $\Gamma$ be split into two non-overlapping relatively open subcurves $\Gamma_1$ and $\Gamma_2$ of nonzero curve measures, i.e., $\Gamma = \ol{\Gamma_1} \cup \ol{\Gamma_2}$ and $\Gamma_1 \cap \Gamma_2 = \emptyset$. Also, $\Gamma_1$ and $\Gamma_2$ are assumed to consist of finite numbers of connected components. Similar to \cite[Sections 2 and 3 of Chapter I]{GiraultFEM}, the following Helmholtz decomposition can be obtained.

\begin{theorem}[Helmholtz decomposition]\label{thm:helmholtz}
For every $\bv \in [L^2(\Omega)]^2$, the $L^2$-orthogonal decomposition $\bv = \grad q + \gradp \psi$ holds, where $q \in \Hone{\Gamma_1}{\Omega}$ is the unique solution to
\begin{equation}\label{eq:weakgrad}
\text{Find } q \in \Hone{\Gamma_1}{\Omega}\col (\grad q, \grad\phi) = (\bv, \grad \phi), \quad \forall \phi \in \Hone{\Gamma_1}{\Omega},
\end{equation}
and $\psi \in \Hone{\Gamma_2}{\Omega}$ is the unique solution to
\begin{equation}\label{eq:weakgradp}
\begin{split}
\text{Find } &\psi \in \Hone{\Gamma_2}{\Omega}\col (\gradp \psi, \gradp\nu) = (\bv, \gradp \nu), \quad \forall \nu \in \Hone{\Gamma_2}{\Omega}.
\end{split}
\end{equation}
\end{theorem}

\begin{remark}\label{rem:stronggrad}
When $\bv \in \Hdiv{\Omega}$, using \eqref{eq:greendiv}, \eqref{eq:weakgrad} can be expressed as: Find $q \in \Hone{\Gamma_1}{\Omega}$ such that $(\grad q, \grad\phi) = \mathopen{}-(\gradd \bv, \phi) + \li \bv\cdot\bn, \phi \ri_\Gamma$, for all $\phi \in \Hone{\Gamma_1}{\Omega}$. This is interpreted as the weak formulation of the following elliptic PDE for $q$:
\[
\lap q = \gradd \bv \text{ in } \Omega,\quad\quad
q = 0 \text{ on } \Gamma_1,\quad\quad
\pardf{q}{\bn} = \bv\cdot\bn \text{ on } \Gamma_2,
\]
where $\lap \psi = \partial_{tt} \psi + \partial_{xx} \psi$. Similarly, the weak problem \eqref{eq:weakgradp} can be (formally) interpreted as the following elliptic PDE for $\psi$:
\[
-\lap \psi = \curl \bv \text{ in } \Omega,\quad\quad
\psi = 0 \text{ on } \Gamma_2,\quad\quad
\pardf{\psi}{\bn} = -\bv\cdot\btau \text{ on } \Gamma_1,
\]
where $\curl \bv = \partial_t v_2 - \partial_x v_1$, and $\btau = [-n_2,n_1]$ (here, $\bn = [n_1,n_2]$) is the unit tangent to the boundary.
\end{remark}

\section{Scalar hyperbolic balance laws}
\label{sec:balance}

This section provides an overview of the basic notions and properties associated with hyperbolic balance laws. This serves as a foundation for the sections that follow.

In this paper, we consider scalar \emph{hyperbolic balance laws} (see \cite{LeVequeHyperbolic}) of the form
\begin{subequations}\label{eq:balance}
\begin{alignat}{2}
\gradd \bff (u) &= r &&\text{ in } \Omega,\label{eq:balancePDE}\\
u &= g &&\text{ on } \Gamma_I,\label{eq:balanceBC}
\end{alignat}
\end{subequations}
where the flux vector $\bff\col \bbR \to \bbR^2$, $\bff \in [L^\infty_{\loc}(\bbR)]^2$, $\bff = [f_1, f_2]$, the source term $r \in L^2(\Omega)$, the inflow boundary data $g \in L^\infty(\Gamma_I)$ are given, and $u$ is the unknown dependent variable. Recall that $L^\infty_{\loc}(\bbR)$ is the space of measurable functions that are in $L^\infty(J)$, for all compact subsets $J \subset \bbR$. Clearly, under the assumptions on $\bff$ below, \eqref{eq:balancePDE} can be represented as a first-order quasilinear PDE for $u$. When $r \equiv 0$, \eqref{eq:balance} becomes a \emph{hyperbolic conservation law}. Here, $\Gamma_I$ denotes the inflow portion of the boundary to be considered in more detail below. Since the focus is on weak solutions to \eqref{eq:balance} (defined below), we only assume that the components of the flux vector, $\bff$, are locally Lipschitz-continuous on $\bbR$; that is, for every compact subset $J \subset \bbR$, there exists a constant $K_{\bff,J} > 0$, which generally depends on $\bff$ and $J$, such that
\begin{equation}\label{eq:lipschitz}
\lv f_i(\upsilon_1) - f_i(\upsilon_2) \rv \le K_{\bff,J} \lv \upsilon_1 - \upsilon_2 \rv, \quad\forall \upsilon_1,\upsilon_2\in J, \quad i=1, 2.
\end{equation}
Note that, by Rademacher's theorem (see, e.g., \cite{EvansPDE}), \eqref{eq:lipschitz} implies that $\bff$ is differentiable, in the classical sense, almost everywhere (a.e.) in $\bbR$ and $\bff' \in [L^\infty_{\loc}(\bbR)]^2$.

The assumption \eqref{eq:lipschitz} is reasonable and mild since it includes a wide class of problems. In particular, any continuously differentiable $\bff$ satisfies \eqref{eq:lipschitz}. In general, there are PDEs of interest with discontinuous $\bff$; see \cite{1996ConsLawDiscontFlux}. In this paper, for simplicity of the analysis, we concentrate on problems that satisfy \eqref{eq:lipschitz}. Nevertheless, the considered formulations are also sensible in the general case of discontinuous $\bff$. Currently, \eqref{eq:lipschitz} admits discontinuous $\bff'$, which admits a quasilinear PDE with discontinuous coefficients.

\begin{remark}
In view of \cite[Subsection 5.8.2b]{EvansPDE}, the above assumptions on the flux vector, $\bff$, are equivalent to the simple assumption $\bff \in [W^{1,\infty}_{\loc}(\bbR)]^2$.
\end{remark}

\begin{remark}\label{rem:iota}
In \blue{time-dependent} problems of type \eqref{eq:balance} that are of interest, it typically holds that $f_1 \equiv \iota$, where $\iota\col \bbR \to \bbR$ is the identity function $\iota(\upsilon) = \upsilon$, $\upsilon \in \bbR$, \blue{representing the time derivative, $\partial_t u$}. Nevertheless, it is convenient here to consider scalar balance laws in the general form \eqref{eq:balance}.
\end{remark}

The characteristics of \eqref{eq:balancePDE} have directions determined by $\bff'(\hu)$, where $\hu$ is an exact (weak) solution to \eqref{eq:balance} of interest; i.e., in the nonlinear case, the characteristics depend on the solution. This motivates the following definition of portions of the boundary, $\Gamma$, which also depend on the solution, in the nonlinear case.

\begin{definition}\label{def:portionsboundary}
Let $\hu$ be an exact (weak) solution to \eqref{eq:balance} of interest. The \emph{inflow} portion of the boundary of $\Omega$ is defined as (see \cite{OlsonPhDthesis,2005HdivHyperbolic,2001ErrorLS})
\[
\Gamma_I = \set*{\bx\in \Gamma \where \bff'(\hu)\cdot \bn < 0}.
\]
Similarly, the \emph{outflow} portion of the boundary and the portion that is tangential to the flow are, respectively,
\[
\Gamma_O = \set*{\bx\in \Gamma \where \bff'(\hu)\cdot \bn > 0},\quad \Gamma_T = \set*{\bx\in \Gamma \where \bff'(\hu)\cdot \bn = 0}.
\]
The complement (essentially) in $\Gamma$ of $\Gamma_I$ is $\Gamma_C = \Gamma_O \cup \Gamma_T = \set{\bx\in \Gamma \where \bff'(\hu)\cdot \bn \ge 0}$.
\end{definition}

This motivates the consideration of boundary conditions that are posed on the inflow boundary, $\Gamma_I$, in \eqref{eq:balanceBC}. Furthermore, a consistency requirement on the inflow data, $g$, is that $\bff'(g)\cdot \bn < 0$ on $\Gamma_I$. For the purpose of this paper, the portion $\Gamma_I$ of the boundary is considered known.

It is common in practical applications of PDEs of the type \eqref{eq:balance} to consider solutions that are piecewise continuously differentiable (shortly, piecewise $\cC^1$). Clearly, in general, such solutions are not classical. Therefore, we concentrate on the notion of a weak solution that allows for piecewise $\cC^1$ solutions; cf., \cite{GodlewskiHCL,LeVequeHyperbolic,LeVequeHCL,EvansPDE,1960Conservation,1954WeakSols}.

\begin{definition}
A function $\hu \in L^\infty(\Omega)$ is a \emph{weak solution} to \eqref{eq:balance} if it satisfies
\[
- \int_\Omega \bff(\hu)\cdot\grad\phi \dd\bx = \int_\Omega r\phi \dd\bx - \int_{\Gamma_I} \bff(g)\cdot\bn\phi \dd\sigma, \quad\forall\phi\in \cC^1(\ol{\Omega})\where \phi = 0 \text{ on } \Gamma_C.
\]
In terms of the notation in \cref{sec:basic}, by density, this can be equivalently expressed as
\begin{equation}\label{eq:weaksol}
-(\bff(\hu), \grad\phi) = (r, \phi) - \li \bff(g)\cdot\bn, \phi \ri_\Gamma, \quad\forall\phi\in H^1_{0,\Gamma_C}(\Omega).
\end{equation}
\end{definition}

The following lemma is obtained easily from the above definitions; see also \cite{GodlewskiHCL, 2005HdivHyperbolic}.

\begin{lemma}\label{lem:hdiv}
Let $\bff \in [L^\infty_{\loc}(\bbR)]^2$ satisfy \eqref{eq:lipschitz}. It holds that $\bff(\hu) \in \Hdiv{\Omega}$ and $\gradd \bff(\hu) = r$, for any weak solution, $\hu \in L^\infty(\Omega)$, to \eqref{eq:balance}.
\end{lemma}
\begin{proof}
It is easy to see, using \eqref{eq:lipschitz}, that $\bff(v) \in [L^\infty(\Omega)]^2 \subset [L^2(\Omega)]^2$, for every $v  \in L^\infty(\Omega)$. Furthermore, in view of \eqref{eq:weaksol} and \eqref{eq:distdiv}, it holds that $\gradd \bff(\hu) \in L^2(\Omega)$ and $\gradd \bff(\hu) = r \in L^2(\Omega)$, for any weak solution, $\hu \in L^\infty(\Omega)$, to \eqref{eq:balance}.
\end{proof}

In other words, the PDE \eqref{eq:balancePDE} is satisfied by $\hu$ in an $L^2$ sense. Also, in view of \cref{lem:hdiv}, the weak formulation \eqref{eq:weaksol} can be seen, in a sense, as the result of applying the Green's formula \eqref{eq:greendiv} to \eqref{eq:balance}.

Note that \cref{lem:hdiv} holds under quite general assumptions and $\hu$ does not need to be piecewise $\cC^1$. In particular, in view of \cite[Lemma 2.4]{2005HdivHyperbolic}, for any piecewise $\cC^1$ function $v$, using \eqref{eq:lipschitz}, $\bff(v) \in \Hdiv{\Omega}$ is equivalent to the \emph{Rankine\-/Hugoniot jump condition}: $\ldb \bff(v)\cdot\bn \rdb_\sC = 0$ a.e. on $\sC$, for every curve $\sC \subset \Omega$; cf., \cite{2005HdivHyperbolic, GodlewskiHCL}. Here, $\ldb \cdot \rdb_\sC$ is the jump across the curve $\sC$ and $\bn$ denotes the unit normal to $\sC$.

\section{Least-squares formulations}
\label{sec:lsapproach}

This section is devoted to the least-squares principles, related to the weak formulation \eqref{eq:weaksol}, that can be used for deriving finite element methods for balance laws of the form \eqref{eq:balance}. First, an $H^{-1}$-based formulation is discussed. Then, the approach based on the Helmholtz decomposition in \cref{thm:helmholtz}, which is a main focus of this paper, is described.

\subsection{\texorpdfstring{$H^{-1}$}{H\^{}\{-1\}}-based formulation}
\label{ssec:Hmone}

Here, we comment on the relation between the definition \eqref{eq:weaksol}, of a weak solution to \eqref{eq:balance}, and the $H^{-1}$-type spaces defined in \cref{sec:basic}. First, \blue{a slightly modified version of the divergence operator, called here ``duality divergence'', is presented.}

\begin{definition}
For any vector field $\bv \in [L^2(\Omega)]^2$, the ``duality divergence'' operator $[\wgradd] \col [L^2(\Omega)]^2 \allowbreak\to \Hmone{\Gamma_C}{\Omega}$ is defined as $\wgradd \bv = \ell_{\bv} \in \Hmone{\Gamma_C}{\Omega}$, where $\ell_{\bv}(\phi) = -(\bv, \grad \phi)$, for all $\phi \in \Hone{\Gamma_C}{\Omega}$.
\end{definition}

\begin{remark}
\blue{The main differences between the ``duality'' and the ``standard'', $[\gradd]\col \Hdiv{\Omega} \to L^2(\Omega)$ (defined via \eqref{eq:distdiv}), divergence operators is that the ``duality'' version acts on the larger space $[L^2(\Omega)]^2$ and it ignores boundary terms that are present when using the ``standard'' divergence (cf., \eqref{eq:greendiv}).} For the case when $\bv \in \Hdiv{\Omega}$, owing to \eqref{eq:greendiv}, the relation between the ``duality'' and the ``standard'' divergence operators is
\[
[ \wgradd \bv ](\phi) = -(\bv, \grad \phi) = (\gradd \bv, \phi) - \li \bv\cdot\bn, \phi \ri_\Gamma, \quad \forall \phi \in \Hone{\Gamma_C}{\Omega}.
\]
Particularly, if additionally $\bv\cdot\bn = 0$ on $\Gamma_I$, then $\wgradd \bv$ and $\gradd \bv$ can be equated via the standard embedding of $L^2(\Omega)$ into $\Hmone{\Gamma_C}{\Omega}$. Otherwise, $\wgradd \bv$ and $\gradd \bv$ cannot be identified, since $\wgradd \bv$ treats the terms on the inflow boundary, $\Gamma_I$, differently, which is convenient for the discussion that follows. \blue{The operator $\wgradd$ and, accordingly, the space $\Hmone{\Gamma_C}{\Omega}$ allow for more natural considerations below, since they are innately related to the particular Helmholtz decomposition in \cref{thm:helmholtz}, as demonstrated below in \cref{lem:wdivHmone}.}
\end{remark}

Next, consider the linear functional $\ell_d \in \Hmone{\Gamma_C}{\Omega}$, defined as
\begin{equation}\label{eq:defld}
\ell_d(\phi) = (r, \phi) - \li \bff(g)\cdot\bn, \phi \ri_\Gamma, \quad\forall\phi\in \Hone{\Gamma_C}{\Omega},
\end{equation}
where $r$ and $g$ represent the given data in \eqref{eq:balance}. This functional contains all the given data in \eqref{eq:balance} and \eqref{eq:weaksol}, i.e., it contains both the source and inflow boundary data. Then, in view of \eqref{eq:weaksol} and the definition of $\wgradd$, the problem of finding weak solutions to \eqref{eq:balance} is equivalent to the problem of finding solutions, in $L^\infty(\Omega)$, to the equation
\begin{equation}\label{eq:Hmoneweak}
\wgradd \bff(u) = \ell_d,
\end{equation}
where the equality is understood in $\Hmone{\Gamma_C}{\Omega}$ sense (i.e., it is understood as the equality of functionals in $\Hmone{\Gamma_C}{\Omega}$: $[\wgradd \bff(u)](\phi) = \ell_d(\phi)$, for all $\phi \in \Hone{\Gamma_C}{\Omega}$). Equivalently, this can be expressed as $\lV \wgradd \bff(u) - \ell_d \rV_{-1,\Gamma_C} = 0$.

Thus, a natural discrete least-squares formulation is
\begin{equation}\label{eq:Hmoneprinciple}
u^{\smallh} = \argmin_{v^{\smallh} \in \cU^{\smallh}} \lV \wgradd \bff(v^{\smallh}) - \ell_d \rV_{-1,\Gamma_C}^2,
\end{equation}
for a finite element space $\cU^{\smallh} \subset L^\infty(\Omega)$. In the next subsection, a related least-squares principle is considered, based on the Helmholtz decomposition in \cref{thm:helmholtz}, which is more convenient for implementation.

Finally, the following lemma is useful for the discussion below. It, in a sense, already hints at the relationship between the Helmholtz decomposition and the $H^{-1}$-based formulations \eqref{eq:Hmoneweak} and \eqref{eq:Hmoneprinciple}.

\begin{lemma}\label{lem:wdivHmone}
Let $\bv \in [L^2(\Omega)]^2$ have the Helmholtz decomposition (provided by \cref{thm:helmholtz}, with $\Gamma_1 = \Gamma_C$ and $\Gamma_2 = \Gamma_I$) $\bv = \grad q + \gradp \psi$, for the respective unique $q \in \Hone{\Gamma_C}{\Omega}$ and $\psi \in \Hone{\Gamma_I}{\Omega}$. Then it holds that $[\wgradd \bv](\phi) = -(\grad q, \grad \phi)$, for all $\phi \in \Hone{\Gamma_C}{\Omega}$, and $\lV \wgradd \bv \rV_{-1,\Gamma_C} = \lV \grad q \rV$.
\end{lemma}
\begin{proof}
Using the definition of $\wgradd$ and \eqref{eq:weakgrad}, for any $\phi \in \Hone{\Gamma_C}{\Omega}$, it holds that
\begin{gather*}
[\wgradd \bv](\phi) = -(\bv, \grad \phi) = -(\grad q, \grad \phi),\\
\lV \wgradd \bv \rV_{-1,\Gamma_C} = \sup_{\phi \in \Hone{\Gamma_C}{\Omega}}\frac{\lv [\wgradd \bv](\phi) \rv}{\lv \phi \rv_{1}} = \sup_{\phi \in \Hone{\Gamma_C}{\Omega}}\frac{\lv (\grad q, \grad \phi) \rv}{\lv \phi \rv_{1}} = \lV \grad q \rV.
\end{gather*}
\end{proof}

As a consequence of \cref{lem:wdivHmone}, $\lV \wgradd \bv \rV_{-1,\Gamma_C} \le \lV \bv \rV = (\lV \grad q \rV^2 + \lV \gradp \psi \rV^2)^\halff$, for any $\bv \in [L^2(\Omega)]^2$. This shows that $[\wgradd] \col [L^2(\Omega)]^2 \to \Hmone{\Gamma_C}{\Omega}$ is continuous. Also, the null space of $\wgradd$ is
\begin{equation}\label{eq:wdivnull}
\cN(\wgradd) = \set*{\bv \in [L^2(\Omega)]^2 \where \bv = \gradp \nu \text{ for some } \nu \in \Hone{\Gamma_I}{\Omega}}.
\end{equation}

\subsection{Formulation based on the Helmholtz decomposition}
\label{ssec:helmholtz}

In this subsection, the formulation based on the Helmholtz decomposition described in \cref{thm:helmholtz}, which is a main focus of this paper, is presented. Here, \cref{thm:helmholtz} is applied, using the notation $\Gamma_1 = \Gamma_C$ and $\Gamma_2 = \Gamma_I$.

Let $\hu \in L^\infty(\Omega)$ be a weak solution of interest to \eqref{eq:balance}. In view of \cref{lem:hdiv}, $\bff(\hu) \in \Hdiv{\Omega}\subset [L^2(\Omega)]^2$. Thus, consider the Helmholtz decomposition
\begin{equation}\label{eq:fuhelmholtz}
\bff(\hu) = \grad q + \gradp \psi,
\end{equation}
for the respective, uniquely determined (by $\bff(\hu)$), $q \in \Hone{\Gamma_C}{\Omega}$ and $\psi \in \Hone{\Gamma_I}{\Omega}$. Owing to \eqref{eq:weakgrad}, the definition of $\wgradd$, and \eqref{eq:Hmoneweak}, it holds that
\begin{equation}\label{eq:weakq}
(\grad q, \grad\phi) = (\bff(\hu), \grad \phi) = -[\wgradd\bff(\hu)](\phi) = -\ell_d(\phi), \quad \forall \phi \in \Hone{\Gamma_C}{\Omega},
\end{equation}
where $\ell_d$ is defined in \eqref{eq:defld}. Moreover, in view of \cref{rem:stronggrad}, \cref{lem:hdiv}, and \eqref{eq:balance}, \eqref{eq:weakq} is interpreted as the weak formulation of the following elliptic PDE for $q$:
\begin{equation}\label{eq:ellipticforq}
\lap q = r \text{ in } \Omega,\quad\quad
q = 0 \text{ on } \Gamma_C,\quad\quad
\pardf{q}{\bn} = \bff(g)\cdot\bn \text{ on } \Gamma_I.
\end{equation}

\begin{remark}\label{rem:multisol}
It is known (cf., \cite{LeVequeHCL,LeVequeHyperbolic}) that nonlinear PDEs of the type \eqref{eq:balance} can have multiple weak solutions. Recall that $\ell_d$ contains all the given data in \eqref{eq:balance}. Thus, by \eqref{eq:weakq}, $q$ in \eqref{eq:fuhelmholtz} is uniquely determined by the given data as the solution to the \emph{elliptic} PDE \eqref{eq:ellipticforq}. In contrast, $\psi$ in \eqref{eq:fuhelmholtz} is uniquely determined once the weak solution, $\hu$, is fixed and it satisfies \eqref{eq:weakgradp} with $\bv=\bff(\hu)$, but it is not directly determined by the data in \eqref{eq:balance}. Hence, if $\tu\in L^\infty(\Omega)$ is another weak solution to \eqref{eq:balance}, then $\bff(\tu) = \grad q + \gradp \tpsi$, for some $\tpsi \in \Hone{\Gamma_I}{\Omega}$, whereas $q$ is unchanged. In view of \eqref{eq:Hmoneweak} and \eqref{eq:wdivnull}, this is to be expected, since $[\bff(\hu) - \bff(\tu)] \in \cN(\wgradd)$. In theory, another source of non-uniqueness is when $\bff(\hu) = \bff(\tu)$, for $\hu \ne \tu$. However, this cannot happen in practical problems, e.g., in a time-dependent problem the first component of $\bff$ is the identity function, $f_1 \equiv \iota$; see \cref{rem:iota}. Thus, the only practically possible source of non-uniqueness of the weak solution to \eqref{eq:balance} is the potential non-uniqueness of the $\Hone{\Gamma_I}{\Omega}$ component of the decomposition in \eqref{eq:fuhelmholtz}, since it is determined only implicitly by the given data; that is, once the component $q$ in \eqref{eq:fuhelmholtz} is fixed by the data, through the weak problem \eqref{eq:weakq}, any $\psi \in \Hone{\Gamma_I}{\Omega}$ can be selected, as long as the equality \eqref{eq:fuhelmholtz} would hold for some $\hu \in L^\infty(\Omega)$, which is the only constraint on the $\Hone{\Gamma_I}{\Omega}$ component in \eqref{eq:fuhelmholtz} that the PDE \eqref{eq:balance} and the weak formulation \eqref{eq:weaksol} (or its equivalent \eqref{eq:Hmoneweak}) provide.
\end{remark}

\blue{Now, let $q \in \Hone{\Gamma_C}{\Omega}$ be the exact (but as yet not explicitly known) unique solution to \eqref{eq:weakq}. Note that here $q$ is viewed as given, since it can be uniquely characterized and obtained from the given data in \eqref{eq:balance} via the elliptic weak problem \eqref{eq:weakq}. Based on the decomposition \eqref{eq:fuhelmholtz}, considering the solution to \eqref{eq:balance} together with the components of the Helmholtz decomposition \eqref{eq:fuhelmholtz} as unknown functions}, \eqref{eq:balance} is reformulated as the following first-order system of PDEs, for the unknowns $v$, $p$, and $\mu$:
\begin{equation}\label{eq:fosbalance}
\begin{alignedat}{4}
\bff(v) - \grad p - \gradp \mu &= \bzero &&\text{ in } \Omega,\quad\quad\quad p &&= 0 &&\text{ on } \Gamma_C,\\
\grad p &= \grad q &&\text{ in } \Omega,\quad\quad\quad \mu &&= 0 &&\text{ on } \Gamma_I.\\
\end{alignedat}
\end{equation}
\blue{Namely, \eqref{eq:fosbalance} is a system of PDEs for the unknown function (that solves \eqref{eq:balance}) and the Helmholtz decomposition of the flux vector. Observe that $q$ encodes all the given data (the source term and the boundary condition) in \eqref{eq:balance}, which is imposed via the second equation in \eqref{eq:fosbalance}. Let $\hv \in L^\infty(\Omega)$, $\hp\in\Hone{\Gamma_C}{\Omega}$, and $\hmu\in\Hone{\Gamma_I}{\Omega}$ solve \eqref{eq:fosbalance} exactly. Then, $\grad \hp = \grad q$ and $\bff(\hv) = \grad q + \gradp \hmu$ are satisfied. Consequently, owing to \cref{lem:wdivHmone} and \eqref{eq:weakq}, it holds that $\wgradd\bff(\hv) = \ell_d$. Hence, $\hv$ solves \eqref{eq:Hmoneweak} and it is, thus, a weak solution to \eqref{eq:balance}. That is, in view of \cref{lem:wdivHmone}, the divergence equation in \eqref{eq:balance} is addressed via the particular Helmholtz decomposition, which is represented by the first equation in \eqref{eq:fosbalance}.}

\blue{By taking the squared $L^2(\Omega)$ norms of the residuals in \eqref{eq:fosbalance},} the least-squares functional, derived from \eqref{eq:fosbalance}, is
\begin{equation}\label{eq:cF}
\cF(v,p,\mu\where q) = \lV \bff(v) - \grad p - \gradp \mu \rV^2 + \lV \grad p - \grad q \rV^2,
\end{equation}
for $v\in L^\infty(\Omega)$, $p\in \Hone{\Gamma_C}{\Omega}$, and $\mu \in \Hone{\Gamma_I}{\Omega}$. This results in a least-squares principle for the minimization of $\cF$. Simply setting $p = q$ and removing the second term in \eqref{eq:cF} is not a practical option, since $q$ is only given implicitly as the solution to \eqref{eq:weakq} and it cannot be exactly represented in practice. One approach to addressing such a minimization is by reformulating it as a two-stage process, where the first stage obtains an approximation to $\grad q$ by solving \eqref{eq:ellipticforq} and then, using this approximation, the minimization of $\cF$ is addressed in the second stage. Alternatively, owing to \eqref{eq:weakq}, we consider the functional, for $v\in L^\infty(\Omega)$, $p\in \Hone{\Gamma_C}{\Omega}$, and $\mu \in \Hone{\Gamma_I}{\Omega}$,
\begin{equation}\label{eq:hcF}
\begin{split}
\hcF(v,p,\mu\where r, g) &= \lV \bff(v) - \grad p - \gradp \mu \rV^2 + \lV \grad p \rV^2 + 2 \ell_d(p)\\
&= \lV \bff(v) - \grad p - \gradp \mu \rV^2 + \lV \grad p \rV^2 + 2 \lb (r, p) - \li \bff(g)\cdot\bn, p \ri_\Gamma \rb,
\end{split}
\end{equation}
which utilizes only available data and is more convenient for a direct implementation. \blue{The derivation of $\hcF$ is based on \eqref{eq:weakq} and the basic equality $\lV \grad p - \grad q \rV^2 = \lV \grad p \rV^2 - 2(\grad q, \grad p) + \lV \grad q \rV^2$, noticing that dropping $\lV \grad q \rV^2$ does not affect the minimization since it acts as an additive constant to the functional.} The minimization of $\cF$ is equivalent (in terms of minimizers) to the minimization of $\hcF$. Notice that the minimal value of $\cF$ is $0$, whereas the minimal value of $\hcF$ is $-\lV \grad q \rV^2$. Unlike $\cF$, $\hcF$ can be evaluated for any given $(v,p,\mu) \in L^\infty(\Omega)\times \Hone{\Gamma_C}{\Omega}\times \Hone{\Gamma_I}{\Omega}$ using only available information, without requiring explicit knowledge of $q$.

Consider finite element spaces $\cU^{\smallh} \subset L^\infty(\Omega)$, $\cVh{\Gamma_C} \subset \Hone{\Gamma_C}{\Omega}$, and $\cVh{\Gamma_I} \subset \Hone{\Gamma_I}{\Omega}$. In general, these spaces do not need to be on the same mesh or of the same order. Nevertheless, for simplicity of notation, we use $h$ to denote the mesh parameters on all spaces and it is not difficult to generalize the results and formulations in this paper. The discrete least-squares formulation of interest is
\begin{equation}\label{eq:nonquadmin}
\begin{alignedat}{2}
\text{minimize }& \quad&&\cF(v^{\smallh},p^{\smallh},\mu^{\smallh} \where q) \;\text{ or }\; \hcF(v^{\smallh},p^{\smallh},\mu^{\smallh}\where r, g),\\
\text{for }& &&v^{\smallh} \in \cU^{\smallh},\; p^{\smallh} \in \cVh{\Gamma_C},\; \mu^{\smallh} \in \cVh{\Gamma_I}.
\end{alignedat}
\end{equation}
If $(u^{\smallh}, q^{\smallh}, \psi^{\smallh}) \in \cU^{\smallh} \times \cVh{\Gamma_C} \times \cVh{\Gamma_I}$ is a minimizer of \eqref{eq:nonquadmin}, then $u^{\smallh} \in \cU^{\smallh}$ is the obtained approximation of a weak solution to \eqref{eq:balance}.

A common approach to solving problems like \eqref{eq:nonquadmin}, which is tailored to non-quadratic least-squares problems, is the \emph{Gauss-Newton method} (see \cite{DennisNO}), where the system \eqref{eq:fosbalance} is linearized by the Newton method and the resulting linear first-order system is reformulated to a quadratic least-squares principle. For example, let $\mathring{v}$ be a current iterate. The aim is to obtain a next iterate $v$. To this purpose, the equation is linearized around $\mathring{v}$. In general, for a (Fr{\'e}chet) differentiable nonlinear operator $F(v)$, the Newton linearization of the equation $F(v) = 0$ around $\mathring{v}$ is $F(\mathring{v}) + F'(\mathring{v})v_{\smallDelta} = 0$, where $F'(\mathring{v})v_{\smallDelta}$ is, generally, the G{\^a}teaux (i.e., directional) derivative of $F$ at $\mathring{v}$ in the direction $v_{\smallDelta}$. This is an equation for the \emph{update (step)} $v_{\smallDelta}$, which is reformulated as a least-squares method below, where the new iterate is obtained as $v = \mathring{v} + v_{\smallDelta}$. In particular, the Newton linearization of \eqref{eq:fosbalance}, with unknowns $v_{\smallDelta}$, $p_{\smallDelta}$, and $\mu_{\smallDelta}$, is
\[
\begin{alignedat}{4}
\bff'(\mathring{v})v_{\smallDelta} - \grad p_{\smallDelta} - \gradp \mu_{\smallDelta} &= \grad \mathring{p} + \gradp \mathring{\mu} - \bff(\mathring{v}) &&\text{ in } \Omega,\quad\quad\quad p_{\smallDelta} &&= 0 &&\text{ on } \Gamma_C,\\
\grad p_{\smallDelta} &= \grad q - \grad \mathring{p} &&\text{ in } \Omega,\quad\quad\quad \mu_{\smallDelta} &&= 0 &&\text{ on } \Gamma_I,\\
\end{alignedat}
\]
where $q$ is given implicitly as the solution to \eqref{eq:weakq}. The corresponding quadratic least-squares functional, for $v_{\smallDelta}\in L^\infty(\Omega)$, $p_{\smallDelta}\in \Hone{\Gamma_C}{\Omega}$, and $\mu_{\smallDelta} \in \Hone{\Gamma_I}{\Omega}$, is
\begin{equation}\label{eq:cFquad}
\begin{split}
\cF_l(v_{\smallDelta},p_{\smallDelta},\mu_{\smallDelta} \where \mathring{v}, \mathring{p}, \mathring{\mu} \where q) &= \lV \bff'(\mathring{v})v_{\smallDelta} - \grad p_{\smallDelta} - \gradp \mu_{\smallDelta} - \grad \mathring{p} - \gradp \mathring{\mu} + \bff(\mathring{v}) \rV^2\\
&\phantom{=}{}+ \lV \grad p_{\smallDelta} - \grad q + \grad \mathring{p} \rV^2,
\end{split}
\end{equation}
or, alternatively, we consider the quadratic functional
\begin{equation}\label{eq:hcFquad}
\begin{split}
\hcF_l(v_{\smallDelta},p_{\smallDelta},\mu_{\smallDelta} \where \mathring{v}, \mathring{p}, \mathring{\mu} \where r, g) &= \lV \bff'(\mathring{v})v_{\smallDelta} - \grad p_{\smallDelta} - \gradp \mu_{\smallDelta} - \grad \mathring{p} - \gradp \mathring{\mu} + \bff(\mathring{v}) \rV^2\\
&\phantom{=}{} + \lv \mathring{p} + p_{\smallDelta} \rv_1^2 + 2 \lb (r, \mathring{p} + p_{\smallDelta}) - \li \bff(g)\cdot\bn, \mathring{p} + p_{\smallDelta} \ri_\Gamma \rb.
\end{split}
\end{equation}
Thus, for current iterates $\mathring{u}^{\smallh} \in \cU^{\smallh}$, $\mathring{q}^{\smallh}\in \cVh{\Gamma_C}$, and $\mathring{\psi}^{\smallh} \in \cVh{\Gamma_I}$, the quadratic discrete least-squares problem is
\begin{equation}\label{eq:quadmin}
\begin{alignedat}{2}
\text{minimize }& \quad&&\cF_l(u_{\smallDelta}^{\smallh},q_{\smallDelta}^{\smallh},\psi_{\smallDelta}^{\smallh} \where \mathring{u}^{\smallh}, \mathring{q}^{\smallh}, \mathring{\psi}^{\smallh} \where q) \;\text{ or }\; \hcF_l(u_{\smallDelta}^{\smallh},q_{\smallDelta}^{\smallh},\psi_{\smallDelta}^{\smallh} \where \mathring{u}^{\smallh}, \mathring{q}^{\smallh}, \mathring{\psi}^{\smallh} \where r, g),\\
\text{for }& &&u_{\smallDelta}^{\smallh} \in \cU^{\smallh},\; q_{\smallDelta}^{\smallh} \in \cVh{\Gamma_C},\; \psi_{\smallDelta}^{\smallh} \in \cVh{\Gamma_I}.
\end{alignedat}
\end{equation}

Since it utilizes only given data, the computationally feasible weak formulation\footnote{It can be derived directly, using the functional $\hcF_l$, or using the functional $\cF_l$ together with \eqref{eq:weakq}.} associated with \eqref{eq:quadmin} is: Find $(u_{\smallDelta}^{\smallh},q_{\smallDelta}^{\smallh},\psi_{\smallDelta}^{\smallh}) \in \cU^{\smallh} \times \cVh{\Gamma_C} \times \cVh{\Gamma_I}$ such that
\begin{small}
\[
\left\{
\begin{alignedat}{9}
&\Big( \bff'(\mathring{u}^{\smallh})u_{\smallDelta}^{\smallh} -{} &&\grad q_{\smallDelta}^{\smallh} -{} &&\gradp \psi_{\smallDelta}^{\smallh},\;{} &&\bff'(\mathring{u}^{\smallh})v^{\smallh} \Big) &&{}= -&&\Big( \bff(\mathring{u}^{\smallh}) -{} &&\grad \mathring{q}^{\smallh} -{} &&\gradp \mathring{\psi}^{\smallh},\; &&\bff'(\mathring{u}^{\smallh})v^{\smallh} \Big),\\
-&\Big( \bff'(\mathring{u}^{\smallh})u_{\smallDelta}^{\smallh} -{} 2&&\grad q_{\smallDelta}^{\smallh} -{} &&\gradp \psi_{\smallDelta}^{\smallh},\;{} &&\grad p^{\smallh} \Big) &&{}= &&\Big( \bff(\mathring{u}^{\smallh}) -{} 2 &&\grad \mathring{q}^{\smallh} -{} &&\gradp \mathring{\psi}^{\smallh},\; &&\grad p^{\smallh} \Big) - \ell_d(p^{\smallh}),\\
-&\Big( \bff'(\mathring{u}^{\smallh})u_{\smallDelta}^{\smallh} -{} &&\grad q_{\smallDelta}^{\smallh} -{} &&\gradp \psi_{\smallDelta}^{\smallh},\;{} &&\gradp\mu^{\smallh} \Big) &&{}= &&\Big( \bff(\mathring{u}^{\smallh}) -{} &&\grad \mathring{q}^{\smallh} -{} &&\gradp \mathring{\psi}^{\smallh},\; &&\gradp\mu^{\smallh} \Big),
\end{alignedat}
\right.
\]
\end{small}%
for all $(v^{\smallh},p^{\smallh},\mu^{\smallh}) \in \cU^{\smallh} \times \cVh{\Gamma_C} \times \cVh{\Gamma_I}$. The final approximation is obtained by iteratively repeating the above procedure. In practice, as in \cref{sec:numerical}, a \emph{damped} Gauss-Newton approach is preferred by combining the Gauss-Newton method with a \emph{line search} algorithm, where the ability to evaluate the functional $\hcF$ is utilized; see \cite{DennisNO}.

Similar to \cite{2005HdivHyperbolic}, by viewing, for the moment, $\bff$ as the (nonlinear) map $\bff\col L^\infty(\Omega) \to [L^2(\Omega)]^2$ and using \eqref{eq:lipschitz}, one can show that, for any $\mathring{v} \in L^\infty(\Omega)$, $\bff'(\mathring{v})\col L^\infty(\Omega) \to [L^2(\Omega)]^2$ is a bounded linear map, i.e., $\lV \bff'(\mathring{v}) v \rV \le C_{\bff,\mathring{v}} \lV v \rV_{L^\infty(\Omega)}$, for all $v \in L^\infty(\Omega)$, where the constant $C_{\bff,\mathring{v}} > 0$ can depend on $\bff$ and $\lV \mathring{v} \rV_{L^\infty(\Omega)}$. This demonstrates that $\bff\col L^\infty(\Omega) \to [L^2(\Omega)]^2$ is (Fr{\'e}chet) differentiable on $L^\infty(\Omega)$, which is of basic importance for the applicability of the Gauss-Newton method to the solution of \eqref{eq:nonquadmin}.

The first-order system \eqref{eq:fosbalance} possesses the convenient property that the nonlinearity (i.e., $\bff(v)$) is only in a zeroth-order term (i.e., a term that does not involve differential operators). Furthermore, the utilization of a formulation like \eqref{eq:nonquadmin}, based on the Helmholtz decomposition, is justified by its relation to the weak formulation \eqref{eq:weaksol} and the $H^{-1}$-type least-squares principle \eqref{eq:Hmoneprinciple}, since this relationship provides the desirable numerical conservation properties of \eqref{eq:nonquadmin}; see \cref{sec:analysis}.

The method here is general and can be applied to balance laws of the type \eqref{eq:balance}. Particularly, it can be used for conservation laws, $r \equiv 0$. However, although the approach here is related to the ideas in \cite{2005HdivHyperbolic,OlsonPhDthesis}, developed for conservation laws, it differs from the methods in \cite{2005HdivHyperbolic,OlsonPhDthesis}, when applied to conservation laws. There, the methods are specially tailored to conservation laws utilizing a different Helmholtz decomposition. Namely, for any $\bv \in [L^2(\Omega)]^2$, $\bv = \grad \tq + \gradp \tpsi$, for uniquely determined $\tq \in H^1_0(\Omega)$ and $\tpsi \in H^1(\Omega)/\bbR$. Whence, $\gradd \bv = 0$ if and only if $\bv = \gradp \tpsi$ for some $\tpsi \in H^1(\Omega)$; see \cite{GiraultFEM}. This is applicable to conservation laws, since, in that case, $\bff(\tu)$ is divergence free, for any respective weak solution $\tu$; cf., \cref{lem:hdiv}. A consequence of using the Helmholtz decomposition of \cref{thm:helmholtz} is that $q \in \Hone{\Gamma_C}{\Omega}$ in \eqref{eq:fuhelmholtz} carries all given data (both the source and the inflow data), whereas an equality like $\bff(\tu) = \gradp \tpsi$, specific to conservation laws, only provides $\gradd \bff(\tu) = 0$ and does not contain any information on the inflow boundary condition. Therefore, additional boundary terms are required in \cite{2005HdivHyperbolic,OlsonPhDthesis}. Although a term like $\lV v - g \rV_{\Gamma_I}^2$, where $\lV \cdot \rV_{\Gamma_I}$ denotes the norm in $L^2(\Gamma_I)$, is not well-defined for a general function $v \in L^\infty(\Omega)$, it makes sense for piecewise continuous and, particularly, for finite element functions in place of $v$. 
Thus, in practice, the formulations \eqref{eq:nonquadmin}, \eqref{eq:quadmin} can be augmented with such a boundary term or a scaled version of it, i.e., it can be added to the functionals in \eqref{eq:cF}, \eqref{eq:hcF}, \eqref{eq:cFquad}, and \eqref{eq:hcFquad}. This is utilized in the numerical experiments in \cref{sec:numerical}. Nonetheless, it is not required for the analysis of the formulation, in \cref{sec:analysis}, since, as mentioned, the boundary data is incorporated in the formulation due to the particular Helmholtz decomposition in \cref{thm:helmholtz} and the resulting relation to \eqref{eq:weaksol}.

\section{Analysis}
\label{sec:analysis}

This section is devoted to the analysis of the formulation in \cref{ssec:helmholtz}. We address the important numerical conservation properties of the method. Also, the regularization of the formulation and the $L^2(\Omega)$ convergence of the approximations are discussed. It is convenient to concentrate on the functional $\cF$ in \eqref{eq:cF}. All considerations, with minor changes, carry over to the functional $\hcF$ in \eqref{eq:hcF} since, as discussed in \cref{ssec:helmholtz}, they provide equivalent formulations. The notation in \cref{ssec:helmholtz} is reused here. In particular, $\hu$ denotes a weak solution to \eqref{eq:balance} and the decomposition \eqref{eq:fuhelmholtz} is utilized.

\subsection{Weak solutions and the conservation property}\label{ssec:conservation}

Here, we investigate in more detail the relationship between \eqref{eq:nonquadmin} and the weak formulation \eqref{eq:weaksol}. As a consequence, the numerical conservation property of the method is obtained.

The following approximation bounds are used as assumptions in the results below:
\begin{alignat}{10}
&\exists\, \hu^{\smallh} \in \cU^{\smallh}&&\col \lV \hu^{\smallh} - \hu \rV &&\le C h^{\beta_{\hu}} &&\lV \hu \rV_{\beta_{\hu}},\label{eq:approxu}\\
&\exists\, \hq^{\smallh} \in \cVh{\Gamma_C}&&\col \lV \grad (\hq^{\smallh} - q) \rV &&\le C h^{\beta_q} &&\lV q \rV_{\beta_q+1},\label{eq:approxq}\\
&\exists\, \hpsi^{\smallh} \in \cVh{\Gamma_I}&&\col \lV \grad (\hpsi^{\smallh} - \psi) \rV &&\le C h^{\beta_\psi} &&\lV \psi \rV_{\beta_\psi+1}, \label{eq:approxpsi}
\end{alignat}
where $C>0$ and $\beta_{\hu},\beta_q,\beta_\psi>0$ are constants that do not depend on $h$, the functions $\hu$, $q$, and $\psi$ are defined in \eqref{eq:fuhelmholtz}, and $\lV \cdot \rV_s$, for $s \in \bbR_+$, denotes the norm in the Sobolev space $H^s(\Omega)$. The assumptions \cref{eq:approxu,eq:approxq,eq:approxpsi} are associated with known interpolation bounds of polynomial approximation theory; see \cite{BrennerFEM}. The approximation orders, $\beta_{\hu}$, $\beta_q$, and $\beta_\psi$, depend on the smoothness of $\hu$, $q$, and $\psi$, respectively, and the polynomial orders of the respective finite element spaces, $\cU^{\smallh}$, $\cVh{\Gamma_C}$, and $\cVh{\Gamma_I}$.

As mentioned, the formulations \eqref{eq:nonquadmin} and \eqref{eq:Hmoneprinciple} are related as they both are based on the notion of a weak solution defined by \eqref{eq:weaksol}. In fact, this relationship can be seen more directly. For $v^{\smallh} \in \cU^{\smallh}$, consider the corresponding Helmholtz decomposition $\bff(v^{\smallh}) = \grad q_{v^{\smallh}} + \gradp \psi_{v^{\smallh}}$, where $q_{v^{\smallh}} \in \Hone{\Gamma_C}{\Omega}$ and $\psi_{v^{\smallh}} \in \Hone{\Gamma_I}{\Omega}$. Then, using \cref{lem:wdivHmone}, the exact decomposition \eqref{eq:fuhelmholtz}, and the $H^\mone$ equation \eqref{eq:Hmoneweak}, it follows
\begin{equation}\label{eq:semidiscrrel}
\cG(v^{\smallh}\where q) = \min_{\substack{p\in \Hone{\Gamma_C}{\Omega}\\ \mu\in \Hone{\Gamma_I}{\Omega}}} \cF(v^{\smallh},p,\mu\where q) = \half \lV \wgradd \bff(v^{\smallh}) - \ell_d \rV_{-1,\Gamma_C}^2.
\end{equation}
Indeed,
\begin{align*}
\cG(v^{\smallh}\where q) &= \min_{p\in \Hone{\Gamma_C}{\Omega}}\lb \lV \grad q_{v^{\smallh}} - \grad p \rV^2 + \lV \grad p - \grad q \rV^2 \rb + \min_{\mu\in \Hone{\Gamma_I}{\Omega}} \lV \gradp \psi_{v^{\smallh}} - \gradp \mu \rV^2\\
&= \min_{p\in \Hone{\Gamma_C}{\Omega}}\lb \lV \grad q_{v^{\smallh}} - \grad p \rV^2 + \lV \grad p - \grad q \rV^2 \rb = \half \lV \grad q_{v^{\smallh}} - \grad q \rV^2\\
&= \half \lV \wgradd[\bff(v^{\smallh}) - \bff(\hu)] \rV_{-1,\Gamma_C}^2 = \half \lV \wgradd \bff(v^{\smallh}) - \ell_d \rV_{-1,\Gamma_C}^2.
\end{align*}
Thus, the minimization \eqref{eq:Hmoneprinciple} is equivalent to the minimization
\begin{alignat*}{2}
\text{minimize }& \quad&&\cF(v^{\smallh},p,\mu \where q) \;\text{ or }\; \hcF(v^{\smallh},p,\mu \where r, g),\\
\text{for }& &&v^{\smallh} \in \cU^{\smallh},\; p \in \Hone{\Gamma_C}{\Omega},\; \mu \in \Hone{\Gamma_I}{\Omega}.
\end{alignat*}

However, in a practical, fully discrete formulation like \eqref{eq:nonquadmin}, finite element spaces $\cVh{\Gamma_C} \subset \Hone{\Gamma_C}{\Omega}$ and $\cVh{\Gamma_I} \subset \Hone{\Gamma_I}{\Omega}$ are utilized. Therefore, a more detailed study of the relationship between \eqref{eq:Hmoneprinciple} and \eqref{eq:nonquadmin} is presented below. It provides the relation between \eqref{eq:nonquadmin} and the weak formulation \eqref{eq:weaksol}, which is important for the numerical conservation properties of formulation \eqref{eq:nonquadmin}. To this end, for simplicity of exposition, the following functional is introduced, for $v \in L^\infty(\Omega)$:
\[
\cG^{\smallh}(v\where q) = \min_{\substack{p^{\smallh}\in\scVh{\Gamma_C}\\ \mu^{\smallh}\in\scVh{\Gamma_I}}} \cF(v, p^{\smallh}, \mu^{\smallh}\where q).
\]
Considering, for $v \in L^\infty(\Omega)$, the corresponding Helmholtz decomposition $\bff(v) = \grad q_{v} + \gradp \psi_{v}$, where $q_{v} \in \Hone{\Gamma_C}{\Omega}$ and $\psi_{v} \in \Hone{\Gamma_I}{\Omega}$, it is easy to see that
\begin{equation}\label{eq:cGhalt}
\cG^{\smallh}(v\where q) = \min_{\substack{p^{\smallh}\in\scVh{\Gamma_C}\\ \mu^{\smallh}\in\scVh{\Gamma_I}}} \lb \lV \grad p^{\smallh} - \grad q_v \rV^2 + \lV \gradp \mu^{\smallh} - \gradp \psi_v \rV^2 + \lV \grad p^{\smallh} - \grad q \rV^2 \rb.
\end{equation}
The minimization in \eqref{eq:cGhalt} is a least-squares problem, where $v$ and $q$ are viewed as given. It is trivial to check that the respective formulation is $[\Hone{\Gamma_C}{\Omega}\times \Hone{\Gamma_I}{\Omega}]$-equivalent, implying the existence and uniqueness of a minimizer. Thus, the functional $\cG^{\smallh}$ is well-defined. Also, problem \eqref{eq:nonquadmin} can be equivalently expressed as
\begin{equation}\label{eq:cGhmin}
u^{\smallh} = \argmin_{v^{\smallh}\in\cU^{\smallh}} \cG^{\smallh}(v^{\smallh}\where q).
\end{equation}

The relationship between \eqref{eq:Hmoneprinciple} and \eqref{eq:nonquadmin} (or \eqref{eq:cGhmin}) as well as other properties of the formulation \eqref{eq:nonquadmin} (or \eqref{eq:cGhmin}) are now shown in the following theorem; see \cite{OlsonPhDthesis,2005HdivHyperbolic} for a related discussion on conservation laws.

\begin{theorem}\label{thm:relation}\leavevmode
\begin{enumerate}[label=(\roman*),nolistsep,wide,align=right]
\item\label{thm:Hmonecontrol} For any $v^{\smallh} \in \cU^{\smallh}$, it holds $\frac{1}{2} \lV \wgradd \bff(v^{\smallh}) - \ell_d \rV_{-1,\Gamma_C}^2 \le \cG^{\smallh}(v^{\smallh}\where q)$.
\item\label{anotherlabel} Assume the approximation bounds \eqref{eq:approxq} and \eqref{eq:approxpsi}. For $v^{\smallh} \in \cU^{\smallh}$, consider the corresponding Helmholtz decomposition $\bff(v^{\smallh}) = \grad q_{v^{\smallh}} + \gradp \psi_{v^{\smallh}}$. Then, using the notation in \eqref{eq:fuhelmholtz} and \eqref{eq:defld}, the following estimates hold, for some constant $C>0$:
\begin{align}
\cG^{\smallh}(v^{\smallh}\where q) &\le 2\lV \wgradd \bff(v^{\smallh}) - \ell_d \rV_{-1,\Gamma_C}^2 + C h^{2\beta_q}\lV q \rV_{\beta_q+1}^2 + \textstyle{\min_{\mu^{\smallh} \in \scVh{\Gamma_I}}} \lV \grad (\mu^{\smallh} - \psi_{v^{\smallh}}) \rV^2,\nn\\
\cG^{\smallh}(v^{\smallh}\where q) &\le 2\lV \wgradd \bff(v^{\smallh}) - \ell_d \rV_{-1,\Gamma_C}^2 + C h^{2\beta_q}\lV q \rV_{\beta_q+1}^2 + C h^{2\beta_{\psi_{h}}}\lV \psi_{v^{\smallh}} \rV_{\beta_{\psi_{h}}+1}^2,\nn\\
&\cG^{\smallh}(v^{\smallh}\where q) \le 2\lV \bff(v^{\smallh}) - \bff(\hu) \rV^2 + C h^{2\beta_q}\lV q \rV_{\beta_q+1}^2 + C h^{2\beta_\psi}\lV \psi \rV_{\beta_\psi+1}^2,\label{eq:almostL2}
\end{align}
where we consider an approximation bound similar to \eqref{eq:approxpsi}, but for the function $\psi_{v^{\smallh}}$, providing an estimate with respective approximation order $\beta_{\psi_{h}} > 0$.
\end{enumerate}
\end{theorem}
\begin{remark}\label{rem:betterrate}
Note that the bounds \cref{eq:approxu,eq:approxq,eq:approxpsi} are for some a priori fixed weak solution, $\hu$, and they are invariant with respect to the arguments of the functionals above. In contrast, a similar bound for $\psi_{v^{\smallh}}$, with order $\beta_{\psi_{h}}$, depends on the argument, $v^{\smallh}$, of the functional.
\end{remark}
\begin{proof}[Proof of \cref{thm:relation}]
\ref{thm:Hmonecontrol} By \eqref{eq:semidiscrrel} and the definition of $\cG^{\smallh}$,
\[
\frac{1}{2} \lV \wgradd \bff(v^{\smallh}) - \ell_d \rV_{-1,\Gamma_C}^2 = \min_{\substack{p\in \Hone{\Gamma_C}{\Omega}\\ \mu\in \Hone{\Gamma_I}{\Omega}}} \cF(v^{\smallh},p,\mu\where q) \le \cG^{\smallh}(v^{\smallh}\where q).
\]

\ref{anotherlabel} Recall that $\hu$ is a weak solution to \eqref{eq:balance}, i.e., it solves the equation \eqref{eq:Hmoneweak}. By \eqref{eq:cGhalt}, \cref{lem:wdivHmone}, \eqref{eq:fuhelmholtz}, \eqref{eq:Hmoneweak}, \eqref{eq:approxq}, and the obvious $\lV \gradp \psi \rV = \lV \grad \psi \rV$, it follows that
\begin{align*}
\cG^{\smallh}(v^{\smallh}\where q) &= \min_{\substack{p^{\smallh}\in\scVh{\Gamma_C}\\ \mu^{\smallh}\in\scVh{\Gamma_I}}} [ \lV \grad (p^{\smallh} - q) + \grad (q - q_{v^{\smallh}}) \rV^2 + \lV \gradp (\mu^{\smallh} - \psi_{v^{\smallh}}) \rV^2 + \lV \grad (p^{\smallh} - q) \rV^2 ]\\
&\le 2\lV \grad (q_{v^{\smallh}} - q) \rV^2 + \min_{\mu^{\smallh}\in\scVh{\Gamma_I}} \lV \gradp (\mu^{\smallh} - \psi_{v^{\smallh}}) \rV^2 + 3\min_{p^{\smallh}\in\scVh{\Gamma_C}}\lV \grad (p^{\smallh} - q) \rV^2\\
&\le 2\lV \wgradd \bff(v^{\smallh}) - \ell_d \rV_{-1,\Gamma_C}^2 + \min_{\mu^{\smallh}\in\scVh{\Gamma_I}} \lV \grad (\mu^{\smallh} - \psi_{v^{\smallh}}) \rV^2 + C h^{2\beta_q}\lV q \rV_{\beta_q+1}^2\\
&\le 2\lV \wgradd \bff(v^{\smallh}) - \ell_d \rV_{-1,\Gamma_C}^2 + C h^{2\beta_q}\lV q \rV_{\beta_q+1}^2 + C h^{2\beta_{\psi_{h}}}\lV \psi_{v^{\smallh}} \rV_{\beta_{\psi_{h}}+1}^2.
\end{align*}
Finally, owing to the definition of $\cG^{\smallh}$, \eqref{eq:approxq}, \eqref{eq:fuhelmholtz}, and \eqref{eq:approxpsi}, it holds that
\begin{align*}
\cG^{\smallh}(v^{\smallh}\where q) &= \min_{\substack{p^{\smallh} \in \scVh{\Gamma_C}\\ \mu^{\smallh} \in \scVh{\Gamma_I}}} \lb \lV \bff(v^{\smallh}) - \grad p^{\smallh} - \gradp \mu^{\smallh} \rV^2 + \lV \grad (p^{\smallh} - q) \rV^2 \rb\\
&\le \lV \bff(v^{\smallh}) - \grad \hq^{\smallh} - \gradp \hpsi^{\smallh} \rV^2 + \lV \grad (\hq^{\smallh} - q) \rV^2\\
&\le \lV \bff(v^{\smallh}) - \bff(\hu) + \grad (q - \hq^{\smallh}) + \gradp (\psi - \hpsi^{\smallh}) \rV^2 + C h^{2\beta_q}\lV q \rV_{\beta_q+1}^2\\
&\le 2\lV \bff(v^{\smallh}) - \bff(\hu) \rV^2 + C h^{2\beta_q}\lV q \rV_{\beta_q+1}^2 + C h^{2\beta_\psi}\lV \psi \rV_{\beta_\psi+1}^2,
\end{align*}
where $\hq^{\smallh} \in \cVh{\Gamma_C}$ and $\hpsi^{\smallh} \in \cVh{\Gamma_I}$ satisfy the bounds \eqref{eq:approxq} and \eqref{eq:approxpsi}, respectively.
\end{proof}

The ``control'' of the functional by the $L^2(\Omega)$ norm and the functional convergence are shown next in the following corollaries.

\begin{corollary}\label{cor:L2}
Assume the approximation bounds \eqref{eq:approxq} and \eqref{eq:approxpsi}. Consider a subset $\cQ^{\smallh} \subset \cU^{\smallh}$ that is bounded in $L^\infty(\Omega)$, i.e., there is a constant $B > 0$ such that $\lV \hu \rV_{L^\infty(\Omega)} \le B$ and $\lV v^{\smallh} \rV_{L^\infty(\Omega)} \le B$, for all $v^{\smallh} \in \cQ^{\smallh}$. Then, for some constants $C>0$ and $C_{\bff,B}>0$, where $C_{\bff,B}$ generally depends on $\bff$ and $B$, it holds that
\[
\cG^{\smallh}(v^{\smallh}\where q) \le C_{\bff,B} \lV v^{\smallh} - \hu \rV^2 + C h^{2\beta_q}\lV q \rV_{\beta_q+1}^2 + C h^{2\beta_\psi}\lV \psi \rV_{\beta_\psi+1}^2.
\]
\end{corollary}
\begin{proof}
Consider the compact interval $J = [-B, B]$ in \eqref{eq:lipschitz}. By \eqref{eq:almostL2} and \eqref{eq:lipschitz},
\begin{align*}
\cG^{\smallh}(v^{\smallh}\where q) &\le 2\lV \bff(v^{\smallh}) - \bff(\hu) \rV^2 + C h^{2\beta_q}\lV q \rV_{\beta_q+1}^2 + C h^{2\beta_\psi}\lV \psi \rV_{\beta_\psi+1}^2\\
&\le 4K_{\bff,J}^2\lV v^{\smallh} - \hu \rV^2 + C h^{2\beta_q}\lV q \rV_{\beta_q+1}^2 + C h^{2\beta_\psi}\lV \psi \rV_{\beta_\psi+1}^2,
\end{align*}
where $K_{\bff,J}$ is defined in \eqref{eq:lipschitz}.
\end{proof}

\begin{corollary}\label{cor:zerominconv}
Assume the approximation bounds \cref{eq:approxu,eq:approxq,eq:approxpsi}, that \eqref{eq:cGhmin} has a minimizer $u^{\smallh} \in \cU^{\smallh}$, and that $\hu^{\smallh} \in \cU^{\smallh}$, which satisfies \eqref{eq:approxu}, can be selected such that it forms a bounded sequence in $L^\infty(\Omega)$ as $h\to 0$, i.e., there is a constant $B > 0$ such that $\lV \hu \rV_{L^\infty(\Omega)} \le B$ and $\lV \hu^{\smallh} \rV_{L^\infty(\Omega)} \le B$ as $h \to 0$. Then, for some constants $C>0$ and $C_{\bff,B}>0$, where $C_{\bff,B}$ generally depends on $\bff$ and $B$, it holds that
\[
\cG^{\smallh}(u^{\smallh} \where q) \le C_{\bff,B} h^{2\beta_{\hu}}\lV \hu \rV_{\beta_{\hu}}^2 + C h^{2\beta_q}\lV q \rV_{\beta_q+1}^2 + C h^{2\beta_\psi}\lV \psi \rV_{\beta_\psi+1}^2.
\]
In particular, this implies that $\cG^{\smallh}(u^{\smallh} \where q) \to 0$ as $h \to 0$.
\end{corollary}
\begin{proof}
Similar to \cref{cor:L2}, using \eqref{eq:cGhmin}, \eqref{eq:almostL2}, \eqref{eq:lipschitz}, and \eqref{eq:approxu}, it holds that
\begin{align*}
\cG^{\smallh}(u^{\smallh} \where q) &\le \cG^{\smallh}(\hu^{\smallh} \where q) \le C_{\bff,B} \lV \hu^{\smallh} - \hu \rV^2 + C h^{2\beta_q}\lV q \rV_{\beta_q+1}^2 + C h^{2\beta_\psi}\lV \psi \rV_{\beta_\psi+1}^2\\
&\le C_{\bff,B} h^{2\beta_{\hu}}\lV \hu \rV_{\beta_{\hu}}^2 + C h^{2\beta_q}\lV q \rV_{\beta_q+1}^2 + C h^{2\beta_\psi}\lV \psi \rV_{\beta_\psi+1}^2.
\end{align*}
\end{proof}

\begin{remark}
It is instructive to express the estimate in \cref{cor:zerominconv}, for the corresponding minimizer $(u^{\smallh}, q^{\smallh}, \psi^{\smallh}) \in \cU^{\smallh} \times \cVh{\Gamma_C} \times \cVh{\Gamma_I}$ of \eqref{eq:nonquadmin}, as
\[
\cF(u^{\smallh},q^{\smallh},\psi^{\smallh} \where q) \le C_{\bff,B} h^{2\beta_{\hu}}\lV \hu \rV_{\beta_{\hu}}^2 + C h^{2\beta_q}\lV q \rV_{\beta_q+1}^2 + C h^{2\beta_\psi}\lV \psi \rV_{\beta_\psi+1}^2,
\]
showing that $\cF(u^{\smallh},q^{\smallh},\psi^{\smallh} \where q) \to 0$ (its globally minimal value) as $h \to 0$.
\end{remark}

The above results indicate that asymptotically, as $h\to 0$, the formulations \eqref{eq:Hmoneprinciple} and \eqref{eq:nonquadmin} approach each other. This can be interpreted that, in a sense, \eqref{eq:nonquadmin} behaves like the weak formulation \eqref{eq:weaksol} in the limit, since \eqref{eq:Hmoneprinciple} and \eqref{eq:weaksol} are related. An important consequence of this is the ``numerical conservation'' property of the least-squares formulation \eqref{eq:nonquadmin}, which is the topic of the next theorem. The classical notion of ``numerical conservation'' (or ``conservative schemes'') is associated with the result in \cite{1960Conservation} in the context of hyperbolic conservation laws; cf., \cite{GodlewskiHCL,LeVequeHyperbolic}. It guarantees that, when certain convergence occurs, the limit is a weak solution. As observed in \cite{2005HdivHyperbolic,OlsonPhDthesis}, also in the context of conservation laws, the particular discrete conservation property in the Lax-Wendroff theorem \cite{1960Conservation}, while sufficient, is not necessary for obtaining weak solutions. As in \cite{2005HdivHyperbolic,OlsonPhDthesis}, the considerations here do not fall precisely into the framework introduced in \cite{1960Conservation}, but they are similar in spirit. Namely, having appropriate convergence of the discrete solutions, the limit is guaranteed to be a weak solution to \eqref{eq:balance}. This result largely motivates the consideration of \eqref{eq:nonquadmin}. In fact, it is natural that \eqref{eq:nonquadmin} possesses such a property, since it is related to \eqref{eq:weaksol} by design. This relationship to the notion of a weak solution is associated with the ability of the method to provide correct approximations to piecewise $\cC^1$ (i.e., discontinuous) weak solutions to \eqref{eq:balance}.

\begin{theorem}[numerical conservation]\label{thm:weakconserv}
Let \eqref{eq:cGhmin} possess a minimizer $u^{\smallh} \in \cU^{\smallh}$ (or, equivalently, \eqref{eq:nonquadmin} possess a minimizer $(u^{\smallh}, q^{\smallh}, \psi^{\smallh}) \in \cU^{\smallh} \times \cVh{\Gamma_C} \times \cVh{\Gamma_I}$) and let the assumptions in \cref{cor:zerominconv} hold. Assume, in addition, $L^2(\Omega)$ convergence,
\begin{equation}\label{eq:L2convasm}
\lim_{h\to 0} \lV u^{\smallh} - \tu \rV = 0,
\end{equation}
for some function $\tu \in L^\infty(\Omega)$, and that $u^{\smallh}$ forms a bounded sequence in $L^\infty(\Omega)$ as $h\to 0$, i.e., there is a constant $B > 0$ such that $\lV \tu \rV_{L^\infty(\Omega)} \le B$ and $\lV u^{\smallh} \rV_{L^\infty(\Omega)} \le B$ as $h \to 0$. Then, $\tu$ is a weak solution to \eqref{eq:balance} in the sense of \eqref{eq:weaksol}.
\end{theorem}
\begin{proof}
As in the proof of \cref{cor:L2}, by \eqref{eq:lipschitz}, \eqref{eq:L2convasm} implies that $\lim_{h\to 0} \lV \bff(u^{\smallh}) - \bff(\tu) \rV = 0$. Thus,
\[
(\bff(u^{\smallh}), \grad\phi) \xrightarrow{h\to 0} (\bff(\tu), \grad\phi), \quad\forall\phi\in \Hone{\Gamma_C}{\Omega}.
\]
By \cref{cor:zerominconv} and \cref{thm:relation}\ref{thm:Hmonecontrol}, it holds that $\lim_{h\to 0} \lV \wgradd \bff(u^{\smallh}) - \ell_d \rV_{-1,\Gamma_C} = 0$. This implies, using the definitions of $\wgradd$ and $\ell_d$, in \eqref{eq:defld}, that
\[
-(\bff(u^{\smallh}), \grad\phi) \xrightarrow{h\to 0} (r, \phi) - \li \bff(g)\cdot\bn, \phi \ri_\Gamma, \quad\forall\phi\in \Hone{\Gamma_C}{\Omega}.
\]
Combining the above results provides (cf., \eqref{eq:weaksol})
\[
-(\bff(\tu), \grad\phi) = (r, \phi) - \li \bff(g)\cdot\bn, \phi \ri_\Gamma, \quad\forall\phi\in \Hone{\Gamma_C}{\Omega}.
\]
\end{proof}

The assumptions that $\hu^{\smallh}$ and $u^{\smallh}$ in \cref{cor:zerominconv,thm:weakconserv} are bounded in $L^\infty(\Omega)$ are reasonable, especially when approximating piecewise $\cC^1$ weak solutions to \eqref{eq:balance}. Similar assumptions can be seen in the classical result in \cite{1960Conservation}; see also \cite{GodlewskiHCL}. Additionally, it is easy to see that the convergence assumption \eqref{eq:L2convasm} can be replaced by a convergence in the $L^1(\Omega)$ norm or in a pointwise a.e. sense, like in \cite{1960Conservation,GodlewskiHCL}.

Note that in \cref{thm:weakconserv}, unlike \cite{2005HdivHyperbolic,OlsonPhDthesis}, no special treatment and assumptions associated with the boundary conditions are necessary. This is due to the Helmholtz decomposition in \cref{thm:helmholtz}, which naturally (i.e., in a way related to \eqref{eq:weaksol}) incorporates the boundary conditions into the formulation; see the end of \cref{ssec:helmholtz}.

\subsection{Regularizing the formulation and coercivity in the \texorpdfstring{$L^2$}{L\^{}2} norm}
\label{ssec:reg}

The conservation property in \cref{thm:weakconserv} is natural for the least-squares formulation \eqref{eq:nonquadmin}, since it is related to the weak form \eqref{eq:weaksol} and the $H^{-1}$-based principle \eqref{eq:Hmoneprinciple}. The $L^2(\Omega)$ convergence, like \eqref{eq:L2convasm}, is more challenging to show for formulations closely related to the notion of a weak solution \eqref{eq:weaksol}. Here, we comment on the lack of coercivity, in terms of the $L^2(\Omega)$ norm, of the functional $\cF$ in \eqref{eq:cF} and propose a simple regularization that provides a parameter-dependent $L^2$ coercivity.
The discussion of the $L^2$ convergence properties of the non-regularized method is left for \cref{ssec:conv}.

Notice that the assumptions on the minimizer, $u^{\smallh}$, in \cref{thm:weakconserv} are utilized to obtain the convergence in $[L^2(\Omega)]^2$ of $\bff(u^{\smallh})$ to $\bff(\tu)$, i.e., of the nonlinear term. Conversely, in general, it is reasonable to assume that
\begin{equation}\label{eq:nonlincoerc}
c \lV v_1 - v_2 \rV \le \lV \bff(v_1) - \bff(v_2) \rV,\quad\forall v_1,v_2\in L^\infty(\Omega),
\end{equation}
for some constant $c>0$. Particularly, when time is involved, as mentioned in \cref{rem:iota}, it typically holds $f_1 \equiv \iota$, the identity function on $\bbR$, in which case \eqref{eq:nonlincoerc} is trivial. Thus, under assumption \eqref{eq:nonlincoerc}, $L^2$ convergence of $\bff(u^{\smallh})$ implies $L^2$ convergence of $u^{\smallh}$ and, analogously, the coercivity of the functional in the $[L^2(\Omega)]^2$ norm of $\bff(v)$ implies its coercivity in the $L^2(\Omega)$ norm of $v$. Therefore, the questions of coercivity and convergence of $\cF$ and \eqref{eq:nonquadmin} reduce to the respective properties in terms of $\bff(v)$. 

It is not difficult to observe that a uniform coercivity, which provides the respective control of the $[L^2(\Omega)]^2$ norm of $\bff(v)$, is not an innate property of the functional $\cF$ in \eqref{eq:cF}. Indeed, using  $\bff(v) = \grad q_{v} + \gradp \psi_{v}$, it holds (recall \eqref{eq:semidiscrrel} and \cref{lem:wdivHmone}) that $\cG(v, 0) = \halff \lV \grad q_{v} \rV^2$, whereas $\lV\bff(v)\rV^2 = \lV \grad q_{v} \rV^2 + \lV \gradp \psi_{v} \rV^2$. Thus, the functional can be made small, without making $\lV\bff(v)\rV^2$ small and, hence (due to \eqref{eq:lipschitz}), without making $\lV v \rV^2$ small. That is, the functional $\cF$ provides only partial control of the $[L^2(\Omega)]^2$ norm. Namely, it explicitly controls only the $\Hone{\Gamma_C}{\Omega}$ component, $q_v$, of the Helmholtz decomposition, but not the $\Hone{\Gamma_I}{\Omega}$ component, $\psi_v$. This is associated with the relation to \eqref{eq:weaksol} and \eqref{eq:Hmoneweak}; cf., \eqref{eq:wdivnull} and \cref{rem:multisol}. Moreover, even in the linear case, one can construct an oscillatory counterexample (see \cite{PhDthesis}) to demonstrate the lack of an appropriate uniform coercivity.

The above considerations make it clear that, to obtain $L^2$ coercivity, control of $\gradp \psi_{v}$ needs to be explicitly introduced. Therefore, consider the regularized functional
\begin{equation}\label{eq:reg}
\begin{split}
\cF_\varepsilon(v,p,\mu\where q) &= \cF(v,p,\mu\where q) + \varepsilon^2 \lV \gradp \mu \rV^2\\
&= \lV \bff(v) - \grad p - \gradp \mu \rV^2 + \lV \grad p - \grad q \rV^2 + \varepsilon^2 \lV \gradp \mu \rV^2,
\end{split}
\end{equation}
for some (small) $\varepsilon > 0$, which employs additional explicit $\varepsilon$-dependent control of the $\Hone{\Gamma_I}{\Omega}$ component, $\mu$. The regularization provides $\varepsilon$-dependent coercivity that controls the $[L^2(\Omega)]^2$ norm of the flux, $\bff(v)$, which, in turn, bounds the $L^2(\Omega)$ norm of $v$, owing to \eqref{eq:nonlincoerc}. This is shown, in detail, in the following theorem.

\begin{theorem}[continuity and $\varepsilon$-dependent coercivity in $L^2$]\label{thm:regcoerc}
It holds
\[
\cF_\varepsilon(v,p,\mu\where 0)\le 3 \lp \lV \bff(v) \rV^2 + \lV\grad p\rV^2 \rp + (2+\varepsilon^2) \lV\grad \mu\rV^2,
\]
for $\varepsilon \ge 0$, and
\[
\varepsilon^2\lp \frac{1}{2+3\varepsilon^2} \lV \bff(v) \rV^2 + \frac{1}{2}\lV\grad \mu\rV^2 \rp + \frac{1}{3}\lV\grad p\rV^2 \le \cF_\varepsilon(v,p,\mu\where 0),
\]
for any $\varepsilon > 0$.
\end{theorem}
\begin{proof}
The upper bound follows from the basic inequality, for real numbers, $2ab \le a^2 + b^2$ and $\lV \gradp \mu \rV = \lV \grad \mu \rV$.

To address the coercivity, observe that
\begin{align*}
\lV \bff(v) - \gradp \mu \rV^2 + \lV\grad p\rV^2 &= \lV \bff(v) - \grad p - \gradp \mu + \grad p \rV^2 + \lV\grad p\rV^2\\
&\le 2 \lV \bff(v) - \grad p - \gradp \mu \rV^2 + 3\lV\grad p\rV^2 \le 3\cF(v,p,\mu\where 0).
\end{align*}
Owing to the Young's inequality, it holds, for any $\rho > 0$ and particularly for $\rho = 1 + 3\varepsilon^2/2$,
\begin{align*}
\lV \bff(v) - \gradp \mu \rV^2 + 3\varepsilon^2 \lV\gradp \mu\rV^2 &= \lV \bff(v) \rV^2 + (1 + 3\epsilon^2)\lV \gradp \mu \rV^2 - 2 (\bff(v), \gradp \mu)\\
&\ge \lV \bff(v) \rV^2 + (1 + 3\epsilon^2)\lV \gradp \mu \rV^2 - \frac{\lV \bff(v) \rV^2}{\rho} - \rho\lV \gradp \mu\rV^2\\
&= \frac{\rho - 1}{\rho}\lV \bff(v) \rV^2 + (1 + 3\epsilon^2 - \rho)\lV \gradp \mu \rV^2\\
&= \frac{3\varepsilon^2}{2+3\varepsilon^2}\lV \bff(v) \rV^2 + \frac{3\varepsilon^2}{2}\lV \gradp \mu \rV^2.
\end{align*}
Combining the above inequalities provides the coercivity estimate.
\end{proof}

\begin{remark}
As before, it is instructive, using  $\bff(v) = \grad q_{v} + \gradp \psi_{v}$, to consider the reduced functional
\[
\cG_\varepsilon(v\where 0) = \min_{\substack{p\in \Hone{\Gamma_C}{\Omega}\\ \mu\in \Hone{\Gamma_I}{\Omega}}} \cF_\varepsilon(v,p,\mu\where 0) = \half \lV \grad q_{v} \rV^2 + \frac{\varepsilon^2}{1+\varepsilon^2} \lV \gradp \psi_{v} \rV^2 \ge \frac{\varepsilon^2}{1+\varepsilon^2} \lV \bff(v) \rV^2,
\]
for $\varepsilon \le 1$, obtaining an exact estimate.
\end{remark}

By taking $\varepsilon \to 0$ as $h \to 0$, the properties, especially the numerical conservation, of the formulation, established in \cref{ssec:conservation}, are maintained in the limit, as $h \to 0$, for the regularized version. In particular, let $\varepsilon = h^\eta$, for some $\eta \ge 0$, where increasing $\eta$ weakens the regularization. \blue{In this case, $\varepsilon = h^\eta$, \cref{thm:regcoerc} guarantees an \emph{explicit} mesh-dependent control of the $L^2$ norm, independently of the choice of finite element spaces. In contrast, \cref{ssec:conv} discusses an \emph{implicit} mesh-dependent control of the $L^2$ norm, based on the properties of the combination of finite element spaces used in the (non-regularized) discrete formulation. As demonstrated in \cref{thm:basicconv} below, $L^2$ convergence can be obtained in the presence of mesh-dependent $L^2$ coercivity. Note that such a result holds for both implicit and explicit mesh-dependent coercivity.} In our experience, the regularizing term acts as explicit ``artificial diffusion'' (or ``viscosity'') added to the formulation, i.e., taking $\eta$ close to zero, results in smoother approximate solutions. Moreover, in the extreme case of $\eta = 0$, the regularized functional (weakly) enforces that the flux is a gradient (potential) field, which in essence makes the problem elliptic (a diffusion problem).

The discussion in the previous paragraph is indicative of the potential benefits of regularizing the functional in \eqref{eq:cF} for obtaining an explicit form that is easier to tune and potentially more amenable to establishing properties of the method, like convergence to the physically admissible solution, since the regularization term acts as explicitly added ``vanishing diffusion'' (or ``viscosity''). Moreover, the regularization can be useful for facilitating the construction of efficient preconditioners for the resulting linear systems, which is of important practical value. This is currently an ongoing work; see also \cref{sec:conclusions}.

\subsection{Convergence properties}
\label{ssec:conv}

The theoretical results above show: when $L^2$ convergence holds, the approximations converge to a weak solution; for the non-regularized functional, only convergence in an $H^\mone$-type norm is guaranteed, while $L^2$ convergence is not provided by the innate properties of the functional and, hence, it is more challenging to establish. There are two ways to obtain convergence in the $L^2$ norm: one -- by explicitly regularizing the functional, providing explicit control of the $L^2$ norm; and another -- via a careful choice of the finite element spaces, on which the minimization is performed, providing implicit discrete control of the $L^2$ norm. Here, appropriate properties of the finite element spaces are presented in the form of special inf-sup conditions and how they lead to $L^2$ convergence is shown. This can be viewed as an implicit version of the regularization, where control of the $L^2$ norm is obtained in the discrete setting, despite the fact that such control is not innately a property of the (non-regularized) functional. In view of the discussion in the beginning of \cref{ssec:reg}, this can be seen as providing implicit discrete control of the $\Hone{\Gamma_I}{\Omega}$ component, $\psi_v$, of the Helmholtz decomposition. That is, when the finite element spaces are carefully chosen, explicit regularization is not needed to obtain $L^2$ convergence. Numerical results (\cref{sec:numerical}) demonstrate that the regularized and non-regularized formulations behave similarly, providing desired results -- $L^2$ convergence to the physically admissible weak solution.

Let \eqref{eq:cGhmin} have a minimizer $u^{\smallh} \in \cU^{\smallh}$ and $\bff(u^{\smallh}) = \grad q_{u^{\smallh}} + \gradp \psi_{u^{\smallh}}$. Consider $q \in \Hone{\Gamma_C}{\Omega}$ as defined in \cref{eq:fuhelmholtz,eq:weakq,eq:ellipticforq} and assume, for some $C>0$ and $\varkappa > 0$, that
\begin{equation}\label{eq:qrate}
\lV \wgradd \bff(u^{\smallh}) - \ell_d \rV_{-1,\Gamma_C} = \lV \grad (q_{u^{\smallh}} - q) \rV \le C h^\varkappa.
\end{equation}
\Cref{thm:relation}\ref{thm:Hmonecontrol} and \cref{cor:zerominconv} guarantee that $\varkappa \ge \min\set{\beta_{\hu},\allowbreak \beta_{q},\allowbreak \beta_{\psi}}$. Notice that \eqref{eq:qrate} is essentially an error (or approximation) bound with a rate of convergence $\varkappa$. Further, assume the (semi-discrete) inf-sup condition
\begin{equation}\label{eq:sdinfsup}
\inf_{\bs^{\smallh} \in \bcS^{\smallh}} \sup_{p \in \Hone{\Gamma_C}{\Omega}}\frac{\lv (\bs^{\smallh}, \grad p) \rv}{\lV \bs^{\smallh} \rV \lV \grad p \rV} \ge c h^\alpha,
\end{equation}
for some constants $c,\alpha>0$, where $\bcS^{\smallh} = \myspan\set{\bff(v^{\smallh})\where v^{\smallh} \in \cU^{\smallh}} \subset [L^2(\Omega)]^2$. Particular inf-sup bounds of the type \eqref{eq:sdinfsup}, with $\alpha=1$, are \blue{proved} in \cite{PhDthesis} for the linear case and a few choices of $\cU^{\smallh}$, including standard Lagrangian elements\blue{, which are used in the numerical experiments in \cref{sec:numerical}}.

\begin{remark}
A careful inspection of \cref{thm:relation}\ref{anotherlabel} suggests that the rate of functional decay implied by \eqref{eq:almostL2} and \cref{cor:zerominconv} can be pessimistic. More precisely, in view  of \cref{rem:betterrate}, when $\cU^{\smallh}$ is an $H^1(\Omega)$-conforming space, which is the case for the numerical results in \cref{sec:numerical}, the respective component $\psi_{v^{\smallh}}$ can be sufficiently smooth, so that the decay rate is determined by the smoothness of the component $q \in \Hone{\Gamma_C}{\Omega}$. Therefore, it is reasonable to expect that the actual rate of functional decay would be around $\beta_q$ and $\varkappa \approx \beta_q$. Accordingly, some of the numerical results in \cref{sec:numerical} show faster rates of functional decay than predicted by \cref{cor:zerominconv}.
\end{remark}

The simplest setting that can provide $L^2$ convergence follows. 

\begin{theorem}\label{thm:basicconv}
Let \eqref{eq:qrate} and \eqref{eq:sdinfsup} hold with $\varkappa > \alpha$ and $\cU^{\smallh}$ form an increasing sequence of nested spaces as $h\to 0$. Then, $\bff(u^{\smallh})$ converges, with a rate $\cO(h^{\varkappa - \alpha})$, in $[L^2(\Omega)]^2$ as $h\to 0$.
\end{theorem}
\begin{proof}
Assumption \eqref{eq:qrate} implies that $\lV \grad (q_{u^{\smallh}} - q_{u^{h/2}}) \rV \le C h^\varkappa$. Clearly, from the definition of $\wgradd$, \eqref{eq:sdinfsup} is equivalent to
\begin{equation}\label{eq:discrcoerc}
\lV \wgradd \bs^{\smallh} \rV_{-1,\Gamma_C} \ge c h^\alpha \lV \bs^{\smallh} \rV,\quad\forall \bs^{\smallh} \in \bcS^{\smallh}.
\end{equation}
The nestedness of the $\cU^{\smallh}$ spaces provides $[\bff(u^{\smallh}) - \bff(u^{\smallh/2})] \in \bcS^{\smallh/2}$ and, by \eqref{eq:discrcoerc},
\[
\lV \wgradd [\bff(u^{\smallh}) - \bff(u^{\smallh/2})] \rV_{-1,\Gamma_C} \ge c h^\alpha \lV \bff(u^{\smallh}) - \bff(u^{\smallh/2}) \rV.
\]
The above estimates and \cref{lem:wdivHmone} imply that $\lV \bff(u^{\smallh}) - \bff(u^{\smallh/2}) \rV \le C h^{\varkappa - \alpha}$, which, since $\varkappa > \alpha$, can be used to show that $\bff(u^{\smallh})$ forms a Cauchy sequence in $[L^2(\Omega)]^2$.
\end{proof}

Here, the inf-sup condition \eqref{eq:sdinfsup} is utilized differently compared to the setting of mixed finite element methods \cite{BrennerFEM} or of \cite{2018LLstarAndInverse}. Typically, inf-sup conditions are easily satisfied in the continuous (i.e., infinite-dimensional) case and finite element spaces are appropriately selected to maintain the property in the discrete setting. In our considerations, a continuous version of \eqref{eq:sdinfsup} does not hold since it is essentially the already discussed (\cref{ssec:reg}) uniform $L^2$ coercivity; that is, \eqref{eq:sdinfsup} is a relation that can hold only discretely and fails in the continuous setting. Namely, \eqref{eq:sdinfsup} is an assumption on the ``discrete coercivity'' \eqref{eq:discrcoerc}. Moreover, the condition \eqref{eq:sdinfsup} restrains the space $\cU^{\smallh}$ (based on the flux $\bff$), thus potentially limiting the freedom of choice of $\cU^{\smallh}$.

Notice that assumption \eqref{eq:sdinfsup} involves only the discrete space $\cU^{\smallh}$, without taking into account $\cVh{\Gamma_C}$ and $\cVh{\Gamma_I}$, used in the least-squares formulation \eqref{eq:nonquadmin}. That is, even if $\cVh{\Gamma_C}$ and $\cVh{\Gamma_I}$ approach (or are replaced by) continuous spaces, \cref{thm:basicconv} still guarantees $L^2$ convergence, essentially demonstrating the $L^2$ convergence of the $H^\mone$-based formulation \eqref{eq:Hmoneprinciple}. However, having discrete spaces for the Helmholtz decomposition, especially $\cVh{\Gamma_I}$, can enhance the discrete control of the $L^2$ norm, via enhancing the control of the $\Hone{\Gamma_I}{\Omega}$ component, and thus improving on the $L^2$ convergence of \eqref{eq:nonquadmin}. This is discussed next.

We introduce an inf-sup condition that is more suitable for the fully discrete formulation \eqref{eq:nonquadmin}. First, define the following distance and a corresponding subset of $\bcS^{\smallh}$, of fluxes that are, in a sense, ``close'' to $\cVh{\Gamma_I}$:
\begin{gather*}
R_{u^{\smallh}} = \min_{\mu^{\smallh} \in \scVh{\Gamma_I}} \lV \gradp (\mu^{\smallh} - \psi_{u^{\smallh}}) \rV, \text{ (here, $\bff(u^{\smallh}) = \grad q_{u^{\smallh}} + \gradp \psi_{u^{\smallh}}$)}\\
\bcR^{\smallh} = \set{ \bs^{\smallh} \in \bcS^{\smallh} \where \textstyle{\min_{\mu^{\smallh} \in \scVh{\Gamma_I}}} \lV \gradp (\mu^{\smallh} - \psi_{\bs^{\smallh}}) \rV \le 2 R_{u^{2h}},\text{ where } \bs^{\smallh} = \grad q_{\bs^{\smallh}} + \gradp \psi_{\bs^{\smallh}}},
\end{gather*}
where $u^{\smallh}$ and $u^{2\smallh}$ denote minimizers of \eqref{eq:cGhmin} for respective mesh parameters $h$ and $2h$. Assume the following ``restricted'' version of the inf-sup condition, for some $c,\gamma>0$:
\begin{equation}\label{eq:rdinfsup}
\inf_{\br^{\smallh} \in \bcR^{\smallh}} \sup_{p \in \Hone{\Gamma_C}{\Omega}}\frac{\lv (\br^{\smallh}, \grad p) \rv}{\lV \br^{\smallh} \rV \lV \grad p \rV} \ge c h^\gamma.
\end{equation}
It is trivial that if \eqref{eq:sdinfsup} holds, then \eqref{eq:rdinfsup} also holds and, generally, $\gamma \le \alpha$. \blue{Note that, while there are currently no explicit theoretical results showing cases when $\gamma < \alpha$, \eqref{eq:rdinfsup} represents a stronger version of \eqref{eq:sdinfsup}. It can explain some ``enhanced'' convergence rates in \cref{sec:numerical} (cf., Example 1) that would require stronger control than what \eqref{eq:sdinfsup} can provide (see the discussion following \cref{thm:moreconv} below, see also \cite{PhDthesis}).} The following convergence result is obtained, which is stronger than \cref{thm:basicconv}.

\begin{theorem}\label{thm:moreconv}
Let \eqref{eq:qrate} and \eqref{eq:rdinfsup} hold with $\varkappa > \gamma$ and $\cU^{\smallh}$, $ \cVh{\Gamma_I}$ form respective increasing sequences of nested spaces as $h\to 0$. Assume also that $R_{u^{\smallh}} \le R_{u^{2h}}$, for sufficiently small values of the mesh parameter $h$. Then, $\bff(u^{\smallh})$ converges, with a rate $\cO(h^{\varkappa - \gamma})$, in $[L^2(\Omega)]^2$ as $h\to 0$.
\end{theorem}
\begin{proof}
By \eqref{eq:qrate}, it holds that $\lV \grad (q_{u^{2h}} - q_{u^{\smallh}}) \rV \le C h^\varkappa$. \eqref{eq:rdinfsup} is equivalent to
\begin{equation}\label{eq:rdiscrcoerc}
\lV \wgradd \br^{\smallh} \rV_{-1,\Gamma_C} \ge c h^\gamma \lV \br^{\smallh} \rV,\quad\forall \br^{\smallh} \in \bcR^{\smallh}.
\end{equation}
The nestedness of the $\cU^{\smallh}$ spaces provides $[\bff(u^{2\smallh}) - \bff(u^{\smallh})] \in \bcS^{\smallh}$. Consider
\[
\nu^{2\smallh} = \argmin_{\mu^{2h} \in \cV^{2h}_{\Gamma_I}} \lV \grad (\mu^{2\smallh} - \psi_{u^{2h}}) \rV,\quad \nu^{\smallh} = \argmin_{\mu^{\smallh} \in \scVh{\Gamma_I}} \lV \grad (\mu^{\smallh} - \psi_{u^{\smallh}}) \rV.
\]
The nestedness of the $\cVh{\Gamma_I}$ spaces implies $\nu^{2\smallh} \in \cVh{\Gamma_I}$. Then
\begin{gather*}
\min_{\mu^{\smallh} \in \scVh{\Gamma_I}} \lV \grad (\mu^{\smallh} - (\psi_{u^{2h}} - \psi_{u^{\smallh}})) \rV \le  \lV \grad (\nu^{2\smallh} - \nu^{\smallh} - \psi_{u^{2h}} + \psi_{u^{\smallh}}) \rV\\
\le \lV \grad (\nu^{2\smallh} - \psi_{u^{2h}}) \rV + \lV \grad (\nu^{\smallh} - \psi_{u^{\smallh}}) \rV
= R_{u^{2h}} + R_{u^{\smallh}} \le 2 R_{u^{2h}}.
\end{gather*}
Thus, $[\bff(u^{2\smallh}) - \bff(u^{\smallh})] \in \bcR^{\smallh}$ and, by \eqref{eq:rdiscrcoerc},
\[
\lV \wgradd [\bff(u^{2\smallh}) - \bff(u^{\smallh})] \rV_{-1,\Gamma_C} \ge c h^\gamma \lV \bff(u^{2\smallh}) - \bff(u^{\smallh}) \rV.
\]
The above estimates and \cref{lem:wdivHmone} imply that $\lV \bff(u^{2\smallh}) - \bff(u^{\smallh}) \rV \le C h^{\varkappa - \gamma}$, which, since $\varkappa > \gamma$, can be used to show that $\bff(u^{\smallh})$ forms a Cauchy sequence in $[L^2(\Omega)]^2$.
\end{proof}

\begin{remark}
The assumption $R_{u^{\smallh}} \le R_{u^{2h}}$ in \cref{thm:moreconv} is reasonable since $R_{u^{\smallh}} \to 0$, as $h\to 0$, with a rate determined by a bound like \eqref{eq:approxpsi}.
\end{remark}

\Cref{thm:basicconv,thm:moreconv}, together with \eqref{eq:nonlincoerc}, show that $u^{\smallh}$, a minimizer of \eqref{eq:cGhmin}, converges in $L^2(\Omega)$ to some $\tu\in L^2(\Omega)$. They imply a respective convergence rate of $\cO(h^{\varkappa - \alpha})$ or $\cO(h^{\varkappa - \gamma})$. Under the assumption in \cref{thm:weakconserv} that $u^{\smallh}$ forms a bounded sequence in $L^\infty(\Omega)$, it can be shown that $\tu\in L^\infty(\Omega)$ and, by \cref{thm:weakconserv}, $\tu$ is a weak solution to \eqref{eq:balance}.
The purpose here is to justify that $\lV u^{\smallh} - \tu \rV \to 0$, as $h \to 0$, is reasonable due to control of the $L^2(\Omega)$ norm in the discrete setting, even when a uniform $L^2$ coercivity does not hold, and present basic tools that can aid the analysis of such convergence. The order $\alpha$ in \eqref{eq:sdinfsup} (or the equivalent discrete coercivity \eqref{eq:discrcoerc}) reflects a certain ``weak control'' of the $L^2(\Omega)$ norm. Notice that \eqref{eq:sdinfsup} takes into account the worst case in terms of control and the proper handling of that case is required in \cref{thm:basicconv} to obtain the $L^2$ convergence. In contrast, \cref{thm:moreconv} suggests that handling the globally worst case is not necessary, since, in view of \eqref{eq:cGhalt}, formulation \eqref{eq:nonquadmin} enforces certain ``proximity'' of $\psi_{u^{\smallh}}$ to the discrete space $\cVh{\Gamma_I}$, enhancing the control of the $L^2(\Omega)$ norm. Intuitively, this means that $\cVh{\Gamma_I}$ ``filters out'' modes with high-frequency cross stream oscillations, like the ones in \cite[Section 3.7.1]{PhDthesis}, that can cause the lack of uniform $L^2$ coercivity. That is, the proximity of $\psi_{u^{\smallh}}$ to $\cVh{\Gamma_I}$, measured by $R_{u^{\smallh}}$, enhances the control of the $\Hone{\Gamma_I}{\Omega}$ component of the Helmholtz decomposition \eqref{eq:fuhelmholtz}, while the $\Hone{\Gamma_C}{\Omega}$ component is naturally controlled by $\cF$. Indeed, since $R_{u^h}$ is vanishing, $\bcR^h$ contains fluxes, $\bs^h$, such that their respective $\psi_{\bs^h}$ are essentially contained in the finite element space $\cVh{\Gamma_I}$. This eliminates modes with high frequencies in the cross stream direction, that can be contained in $\bcS^h$, leading to $\gamma$ that is smaller than $\alpha$. Particularly, in the linear case, \cite{PhDthesis} shows that $\alpha = 1$ and argues that smaller values cannot be obtained for the worst case, when considering the entire $\bcS^h$, whereas numerical results (cf., e.g., \cref{sec:numerical}) indicate faster rates of convergence than $\alpha = 1$ would imply, suggesting $\gamma \approx 1/2$. Similarly, the regularization in \eqref{eq:reg} can be intuitively interpreted as an explicit form of eliminating modes with high-frequency cross stream oscillations, which can be manually tuned independently of the choice of spaces.

In summary, the convergence $\lV u^{\smallh} - \tu \rV \to 0$ stem from a complex relationship between the spaces involved in formulation \eqref{eq:nonquadmin} that depends on the flux vector, $\bff$. The ``regularizing'' effect of the finite element spaces for the components of the Helmholtz decomposition, particularly, the effect of the choice of $\cVh{\Gamma_I}$, is demonstrated in the numerical experiments in \cref{sec:numerical}, which also support the ideas of this section that the $L^2$ convergence is due to the balance between the orders of an approximation bound like \eqref{eq:qrate} and a discrete coercivity estimate of the type \eqref{eq:discrcoerc} or \eqref{eq:rdiscrcoerc}.

\section{Numerical results}
\label{sec:numerical}

This section is devoted to numerical results for formulation \eqref{eq:nonquadmin}, utilizing a damped Gauss-Newton procedure, applied to the inviscid Burgers equation, which is of the form \eqref{eq:balance} for $\bff(\upsilon) = [\upsilon, \upsilon^2/2]$, with a discontinuous source term, $r$. The examples are inspired by \cite{2012BurgersSource}, which also provides the exact solutions for computing errors.

\blue{The presented results are obtained using the non-regularized functional. Moreover,} as mentioned in the end of \cref{ssec:helmholtz}, the functional $\hcF$ in \eqref{eq:hcF} is replaced by the following \blue{(non-regularized)} ``augmented'' version (the notation is reused):
\[
\hcF(v^{\smallh},p^{\smallh},\mu^{\smallh} \where r, g) = \lV \bff(v^{\smallh}) - \grad p^{\smallh} - \gradp \mu^{\smallh} \rV^2 + \lV \grad p^{\smallh} \rV^2 + 2 \ell_d(p^{\smallh}) + \lV h^\halff (v^{\smallh} - g) \rV_{\Gamma_I}^2,
\]
where $g$ is given in \eqref{eq:balanceBC} and $\lV\cdot\rV_{\Gamma_I}$ denotes the norm in $L^2(\Gamma_I)$. In all cases, continuous finite element spaces on structured triangular meshes are used. The domain is $\Omega = \set{0 < t < 1,\, -0.25 < x < 1.75}$ and the meshes consist of right-crossed squares, \begin{tikzpicture}[scale=0.25] \draw (0,0) -- (1,1); \draw (0,0) -- (0,1); \draw (0,0) -- (1,0); \draw (0,1) -- (1,1); \draw (1,0) -- (1,1); \end{tikzpicture}, where the coarsest mesh has 16 squares in $t$ and 32 squares in $x$, while the finer meshes are obtained by consecutive uniform refinements.

\blue{Here, $u^{\smallh}$ denotes the obtained approximation, $\hu$ is the exact solution, $M^{\smallh}$ denotes the obtained minimal value of the functional, $\hcF$, on a mesh with a parameter $h$. That is, $M^{\smallh} = \hcF(u^{\smallh},q^{\smallh},\psi^{\smallh} \where r, g)$, where $(u^{\smallh},q^{\smallh},\psi^{\smallh}) = \argmin \hcF(v^{\smallh},p^{\smallh},\mu^{\smallh} \where r, g)$. Since the analytical minimal value of $\hcF$ (which is $-\lV \grad q \rV^2$, where $q$ solves \eqref{eq:weakq}) is not explicitly known during computation, the rate of functional convergence to its minimum as $h\to 0$ is measured via $M^{\smallh} - M^{\smallh/2}$.}

\blue{For completeness, on all convergence graphs, the squared $L^1(\Omega)$ norm of the error is also plotted. In practice, the $L^1(\Omega)$ norm can be viewed as a measure of ``sharpness'' of resolution of discontinuities, since it ``penalizes'' small errors more than the $L^2(\Omega)$ norm. Observe that the $L^1(\Omega)$ norm of the error demonstrates noticeably better behavior (higher convergence rates) compared to the $L^2(\Omega)$ norm, for all examples.}

\paragraph{Example 1 (a single shock)}

\begin{figure}
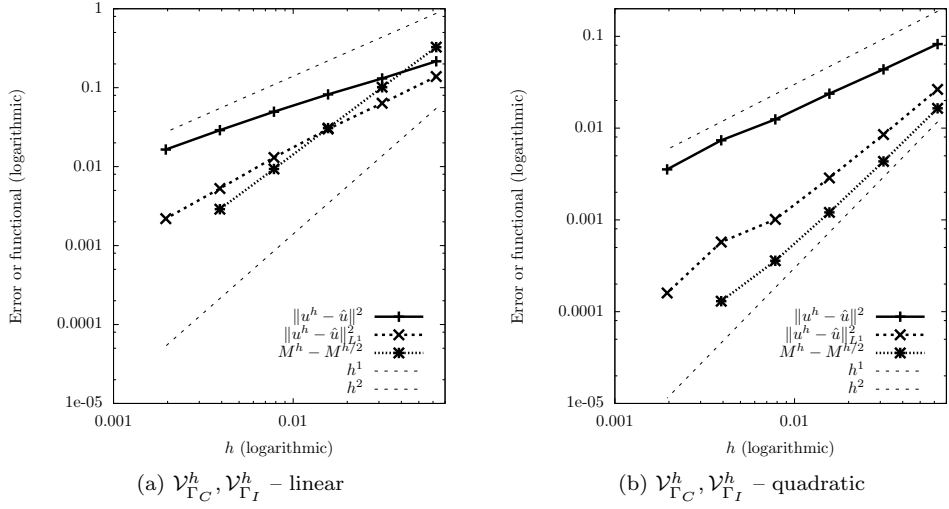

\centering
\subfloat[][$\cVh{\Gamma_C},\cVh{\Gamma_I}$ -- linear]{\resizebox{0.485\textwidth}{!}{\input{1shock_scaled_hhalf_all_linear.texx}}}\quad
\subfloat[][$\cVh{\Gamma_C},\cVh{\Gamma_I}$ -- quadratic]{\resizebox{0.485\textwidth}{!}{\input{1shock_scaled_hhalf_u_linear_rest_quad.texx}}}
\caption[]{Convergence results for Example 1 with linear $\cU^{\smallh}$. \blue{Here, $u^{\smallh}$ denotes the obtained approximation, $\hu$ is the exact solution, $M^{\smallh}$ denotes the obtained minimal value of the functional, $\hcF$, on a mesh with a parameter $h$.}}\label{fig:ex1_conv}
\end{figure}

\begin{figure}
\centering
\subfloat{\includegraphics[width=0.40\textwidth]{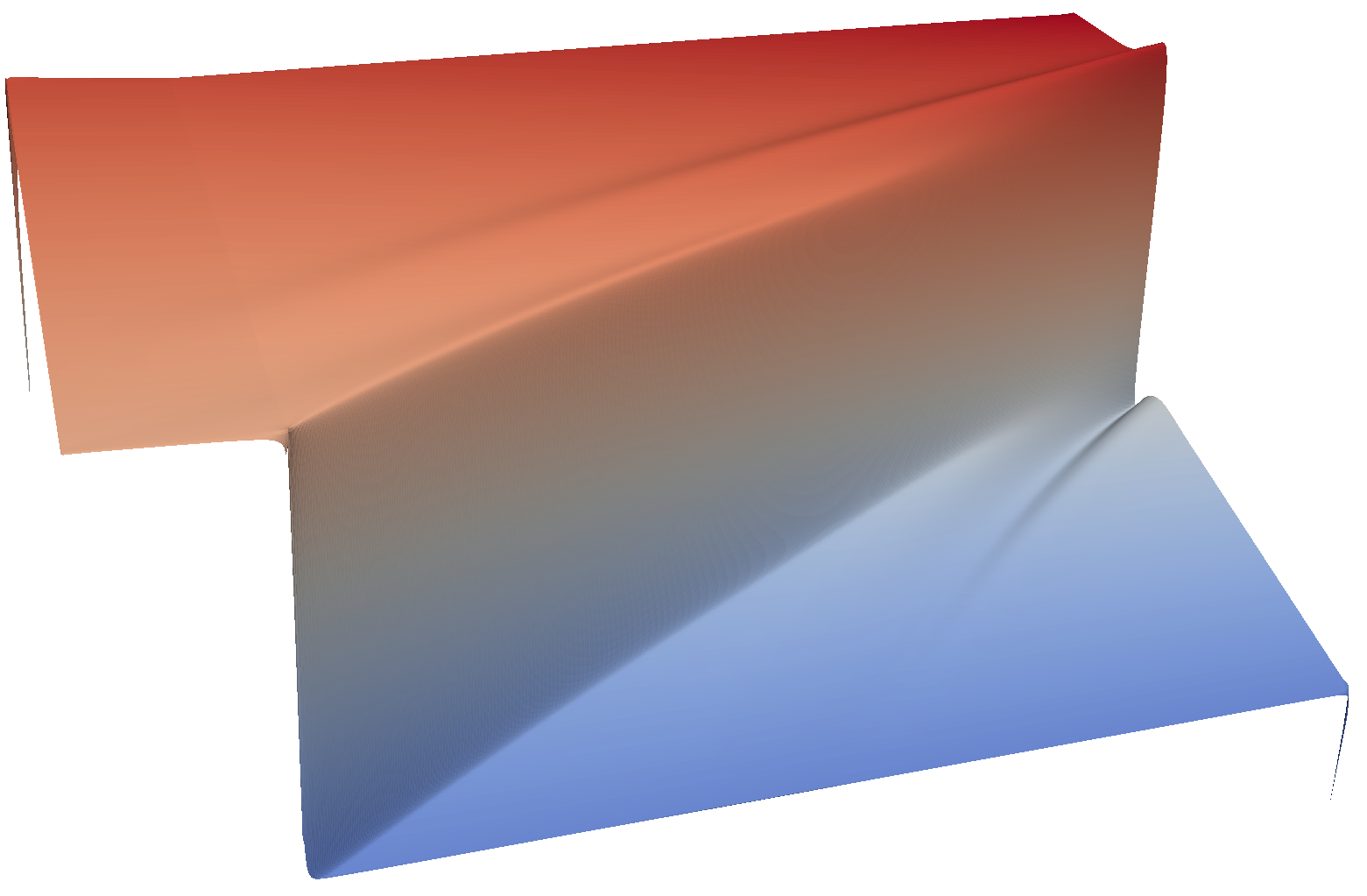}}\quad\quad
\subfloat{\includegraphics[width=0.46\textwidth]{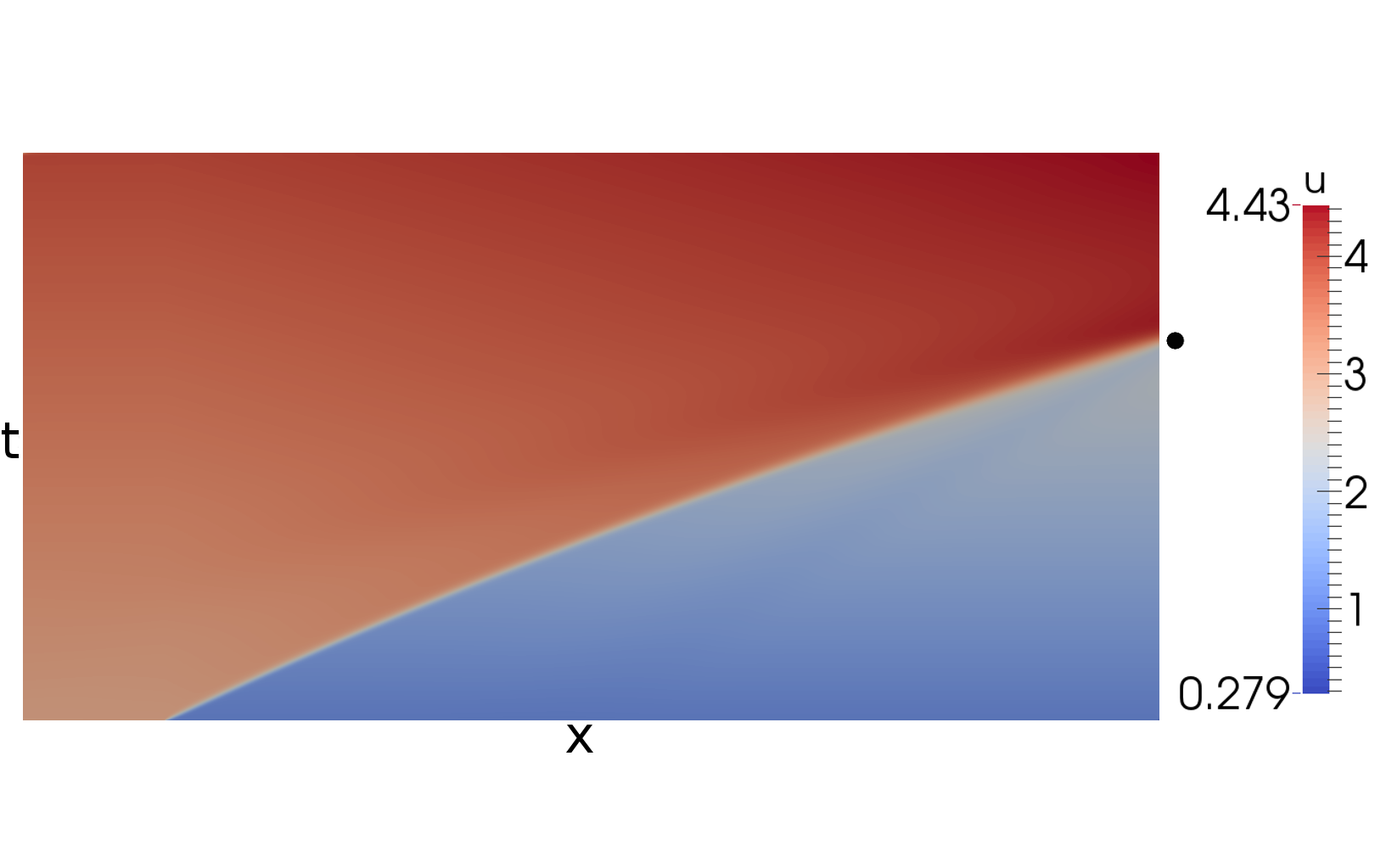}}\\
\subfloat{\includegraphics[width=0.40\textwidth]{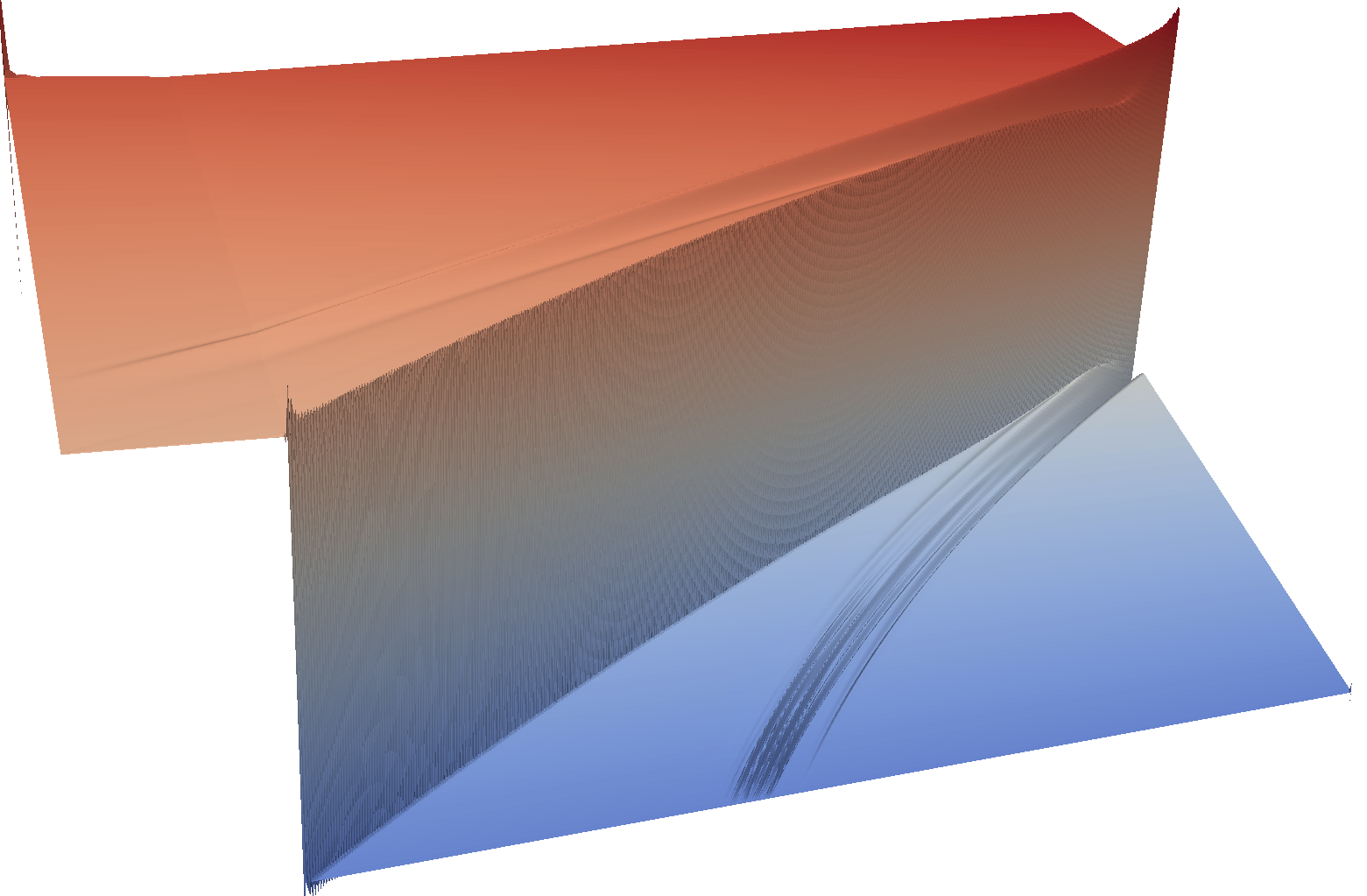}}\quad\quad
\subfloat{\includegraphics[width=0.46\textwidth]{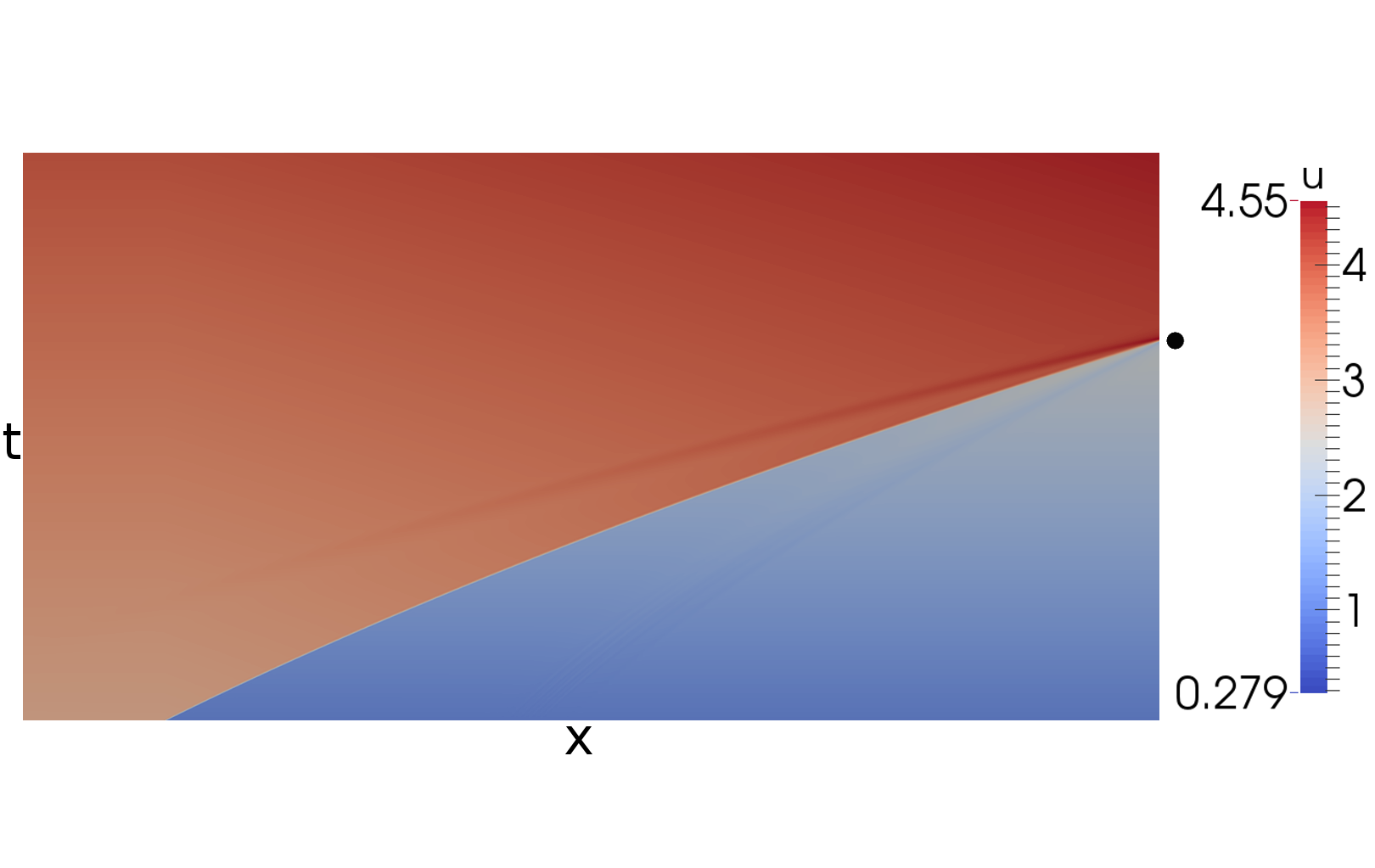}}
\caption[]{The approximation, $u^{\smallh}$, obtained from Example 1 on the finest mesh with linear $\cU^{\smallh}$; $\cVh{\Gamma_C}$,$\cVh{\Gamma_I}$ are linear on the top and quadratic on the bottom. The black dot, $\bullet$, shows where the shock exists the domain in the exact solution, $\hu$.}\label{fig:ex1_sol}
\end{figure}

Consider \eqref{eq:balance} with
\[
r = \begin{cases} 1, &x \le 0\\ 2, &x > 0 \end{cases},\quad\quad g = \begin{cases} 3, &t = 0,\, x \le 0\\ 1, &t = 0,\, x > 0\\ t + 3, &x=-0.25 \end{cases}.
\]
Convergence of the functional values and the approximations obtained by the method are demonstrated in \cref{fig:ex1_conv} for linear and quadratic finite elements. Notice that, in both cases, similar to the methods for conservation laws in \cite{2005HdivHyperbolic,OlsonPhDthesis}, the squared $L^2(\Omega)$ norm of the error approaches $\cO(h)$, which is the theoretically optimal rate \cite{BrennerFEM}. The functional values converge with a higher rate on the tested meshes, similar to \cite{2005HdivHyperbolic,OlsonPhDthesis}. These results align with the discussion in \cref{ssec:conv} that, in general, the functional can only provide a ``weak control'' of the $L^2(\Omega)$ norm and a respective uniform coercivity does not generally hold. Particularly, in terms of \cref{ssec:conv}, \cref{fig:ex1_conv} indicates that $\varkappa \approx 1$ and $\gamma \approx \halff$, providing a rate of $L^2$ convergence around the optimal $\halff$. 

\Cref{fig:ex1_sol} shows the resulting approximations in the two cases. Note that the method correctly captures the shock speed and its curved trajectory, which can be expected considering the convergence in \cref{fig:ex1_conv}. It is worth discussing the spikes in the upper-left and lower-right corners of the domain. Theoretically, such behavior can be linked to the fact that convergence in the $L^2(\Omega)$ norm is not significantly affected by such spikes. A more particular inspection suggests that the spikes can be linked to the specific Helmholtz decomposition and the associated elliptic PDEs in \cref{rem:stronggrad}. Namely, in view of \cref{rem:stronggrad}, the two corners (upper-left and lower-right) with the spikes are precisely where the Neumann and Dirichlet boundary conditions meet in the respective elliptic problems that define the components of the Helmholtz decomposition, resulting in a decreased quality of approximating these components close to the corners, which are important parts of formulation \eqref{eq:nonquadmin}. This is supported by the fact that increasing the order of the spaces for the components of the Helmholtz decomposition, $\cVh{\Gamma_C}$ and $\cVh{\Gamma_I}$, in \cref{fig:ex1_sol}, substantially decreases the spikes, since better approximations of these components are obtained. Observe that the corner spikes do not ``pollute'' the rest of the solution.

Our experience shows that the oscillations around the shock become narrower to accommodate the $L^2(\Omega)$ convergence and remain bounded in amplitude as $h$ decreases. The backward propagation of such oscillations results from formulation \eqref{eq:nonquadmin} being a global (space-time) minimization that currently does not employ any upwind techniques.

Observe, in \cref{fig:ex1_sol}, that the shock is noticeably more smeared for linear $\cVh{\Gamma_C},\cVh{\Gamma_I}$ compared to when they are quadratic and, accordingly, the backward propagating oscillations from the shock exit point are better diffused in the linear case. In our experience, the reduced numerical diffusion in the quadratic case is mostly due to the utilization of higher-order elements for $\cVh{\Gamma_I}$ and not so much due to the space $\cVh{\Gamma_C}$, whereas the reduction of the corner spikes benefits substantially from both $\cVh{\Gamma_I}$ and $\cVh{\Gamma_C}$ being of higher order. This aligns with the discussion about the ``regularizing effect'' of the space $\cVh{\Gamma_I}$ in the end of \cref{ssec:conv}.

\paragraph{Example 2 (a rarefaction wave)}

\begin{figure}
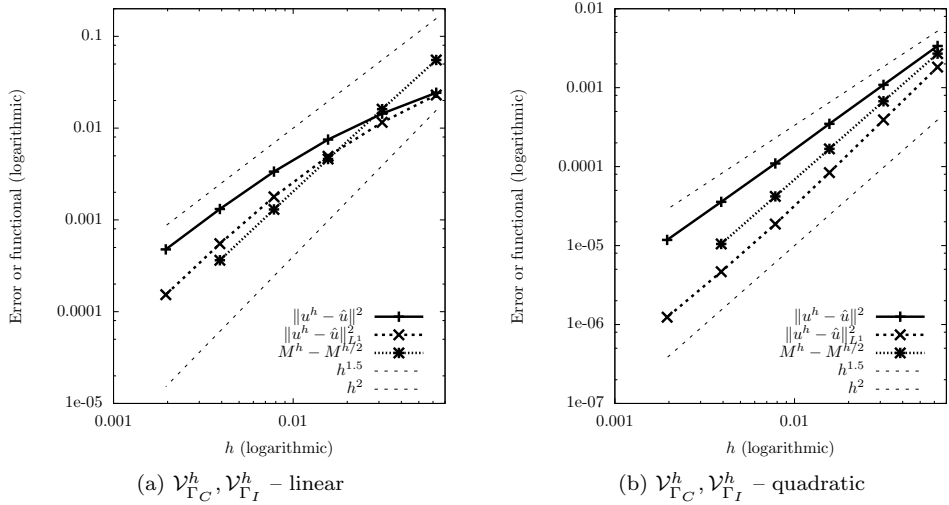

\centering
\subfloat[][$\cVh{\Gamma_C},\cVh{\Gamma_I}$ -- linear]{\resizebox{0.485\textwidth}{!}{\input{rarefaction_scaled_hhalf_all_linear.texx}}}\quad
\subfloat[][$\cVh{\Gamma_C},\cVh{\Gamma_I}$ -- quadratic]{\resizebox{0.485\textwidth}{!}{\input{rarefaction_scaled_hhalf_u_linear_rest_quad.texx}}}
\caption[]{Convergence results for Example 2 with linear $\cU^{\smallh}$. \blue{Here, $u^{\smallh}$ denotes the obtained approximation, $\hu$ is the exact solution, $M^{\smallh}$ denotes the obtained minimal value of the functional, $\hcF$, on a mesh with a parameter $h$.}}\label{fig:ex2_conv}
\end{figure}

\begin{figure}
\centering
\subfloat{\includegraphics[width=0.40\textwidth]{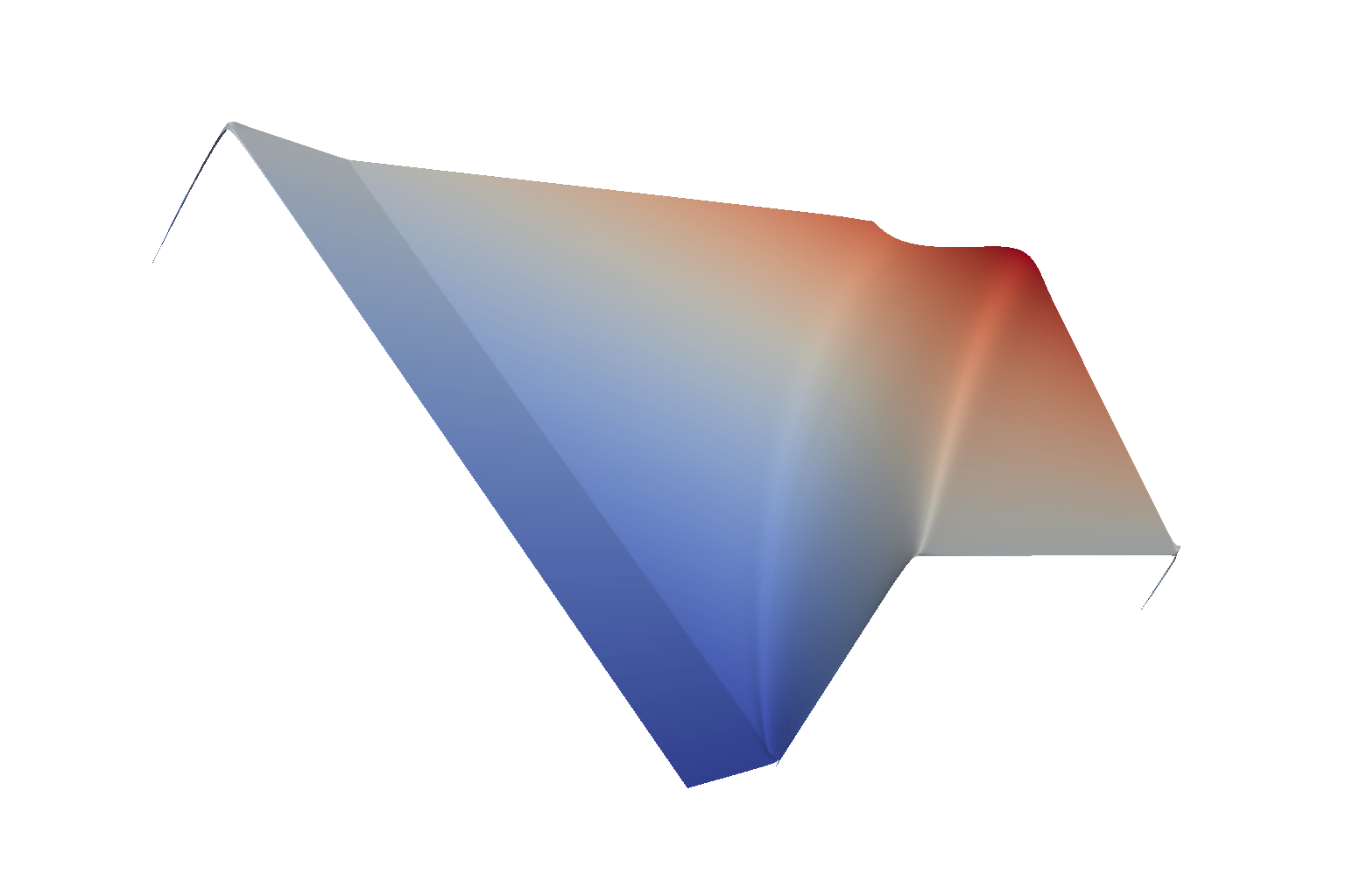}}\quad\quad
\subfloat{\includegraphics[width=0.46\textwidth]{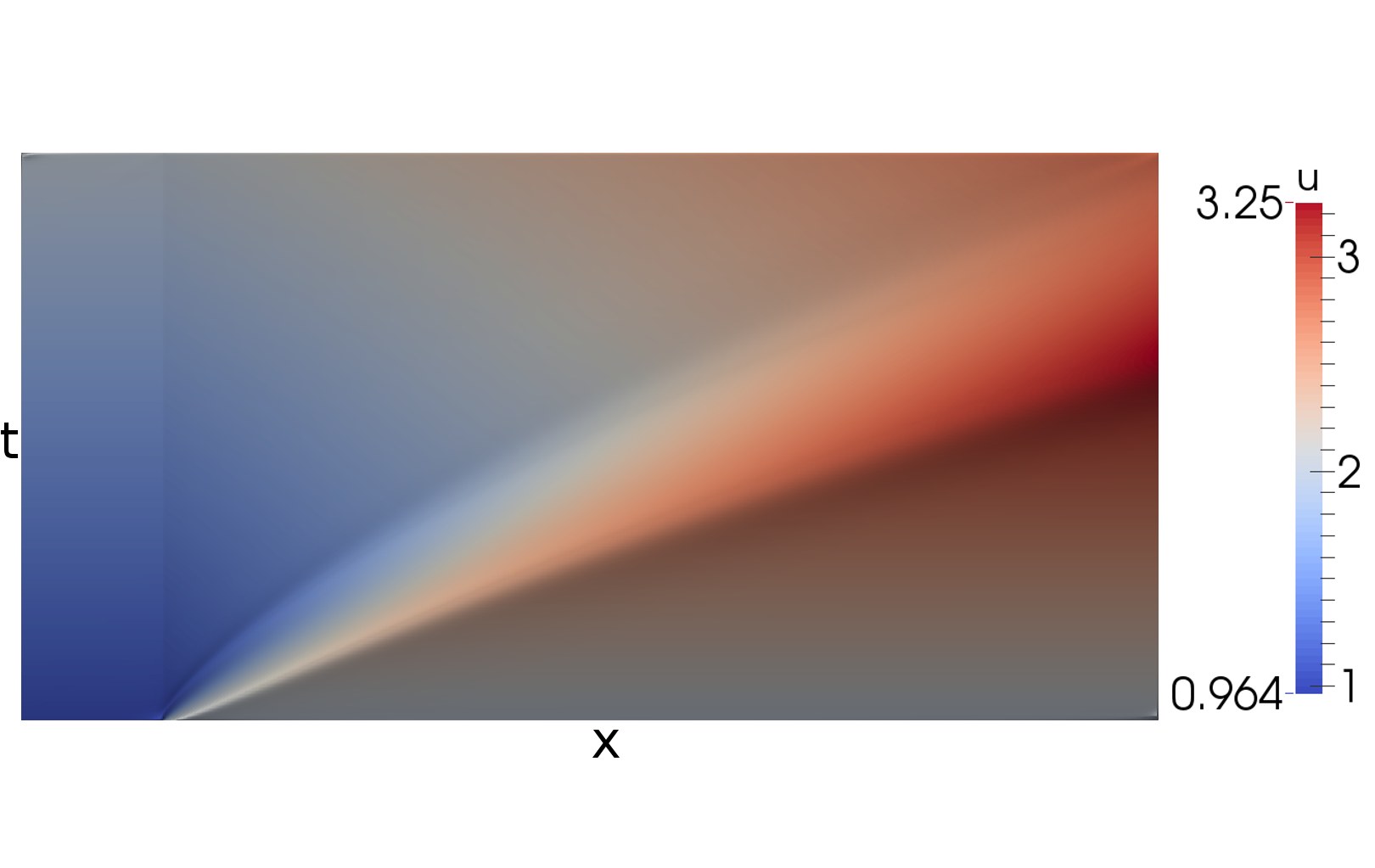}}\\
\subfloat{\includegraphics[width=0.40\textwidth]{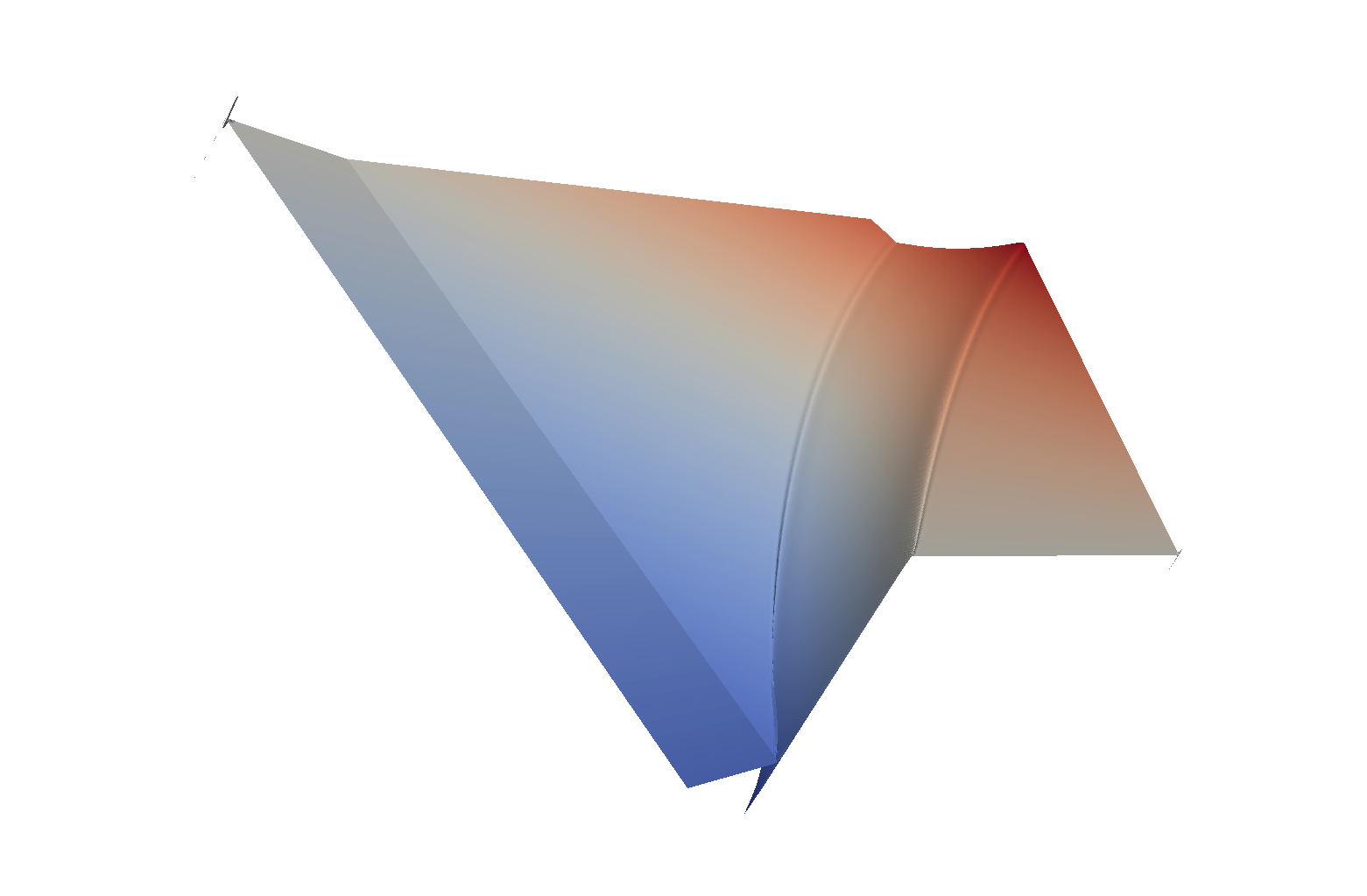}}\quad\quad
\subfloat{\includegraphics[width=0.46\textwidth]{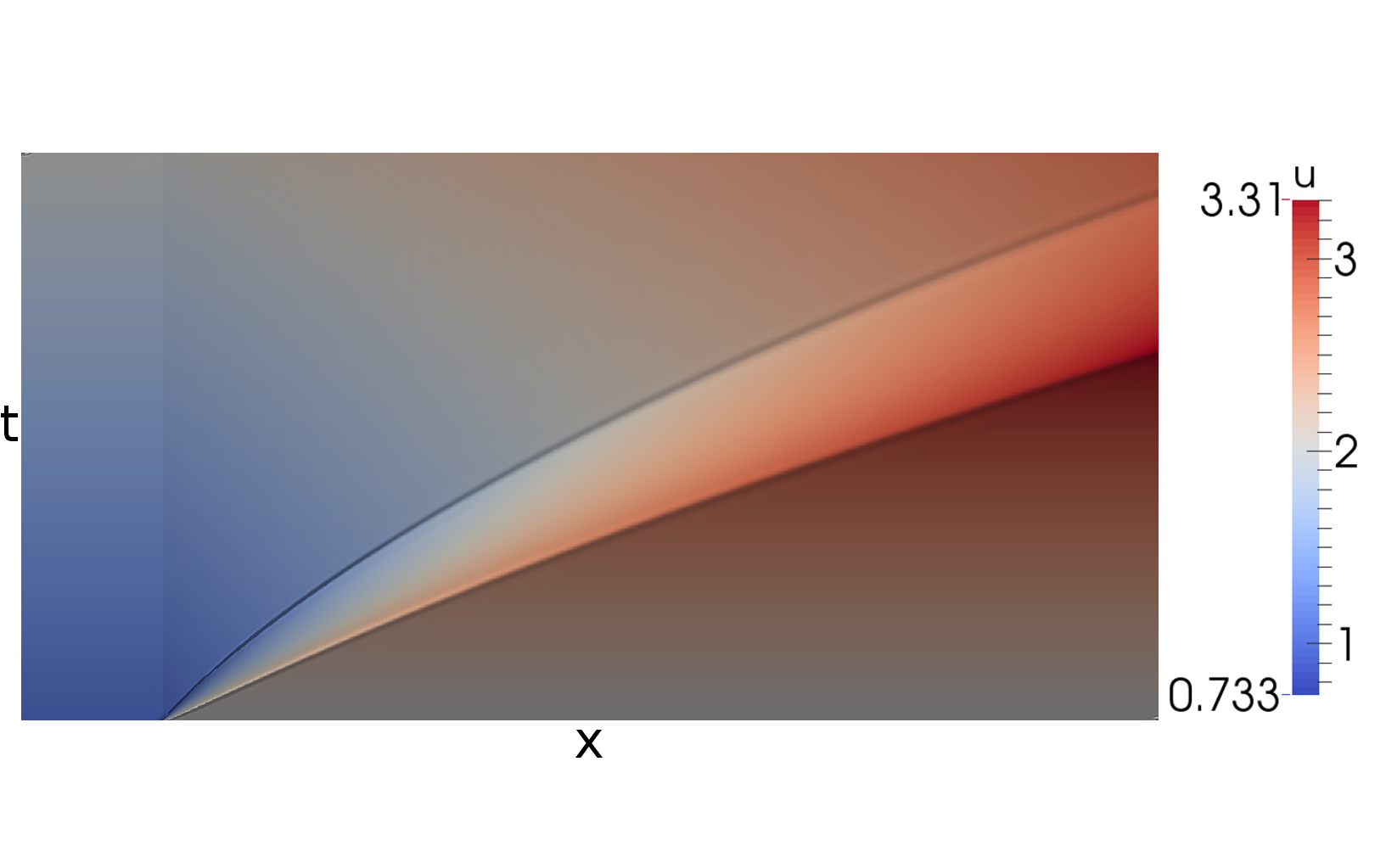}}
\caption[]{The approximation, $u^{\smallh}$, obtained from Example 2 on the finest mesh with linear $\cU^{\smallh}$; $\cVh{\Gamma_C}$,$\cVh{\Gamma_I}$ are linear on the top and quadratic on the bottom.}\label{fig:ex2_sol}
\end{figure}

Consider \eqref{eq:balance} with
\[
r = \begin{cases} 1, &x \le 0\\ 2, &x > 0 \end{cases},\quad\quad g = \begin{cases} 1, &t = 0,\, x \le 0\\ 2, &t = 0,\, x > 0\\ t + 1, &x=-0.25 \end{cases}.
\]
Results are shown in \cref{fig:ex2_conv,fig:ex2_sol}. The main challenge here is that such a setting is associated with an infinite multiplicity of the weak solutions \cite{LeVequeHCL}, where the rarefaction wave (associated with the respective ``characteristic fan'') is the unique physically admissible solution, which is of physical significance. Observe that the method approximates the admissible solution of interest. It is currently unclear if, in theory, this is an innate property of the formulation for all cases. Nevertheless, the numerical diffusion typically present in least-squares methods is indicative of the prospect of obtaining the ``vanishing diffusion'' (or ``viscosity'') solution. The convergence rate is suboptimal here, since the optimal decay rate of the squared $L^2(\Omega)$ norm of the error is $\cO(h^{2-\epsilon})$, for any small $\epsilon > 0$.

\paragraph{Example 3 (colliding shocks)}

\begin{figure}
\centering
\subfloat[][$\cVh{\Gamma_C},\cVh{\Gamma_I}$ -- linear]{\resizebox{0.485\textwidth}{!}{\input{2shocks_scaled_hhalf_all_linear.texx}}}\quad
\subfloat[][$\cVh{\Gamma_C},\cVh{\Gamma_I}$ -- quadratic]{\resizebox{0.485\textwidth}{!}{\input{2shocks_scaled_hhalf_u_linear_rest_quad.texx}}}
\caption[]{Convergence results for Example 3 with linear $\cU^{\smallh}$. \blue{Here, $u^{\smallh}$ denotes the obtained approximation, $\hu$ is the exact solution, $M^{\smallh}$ denotes the obtained minimal value of the functional, $\hcF$, on a mesh with a parameter $h$.}}\label{fig:ex3_conv}
\end{figure}

\begin{figure}
\centering
\subfloat{\includegraphics[width=0.40\textwidth]{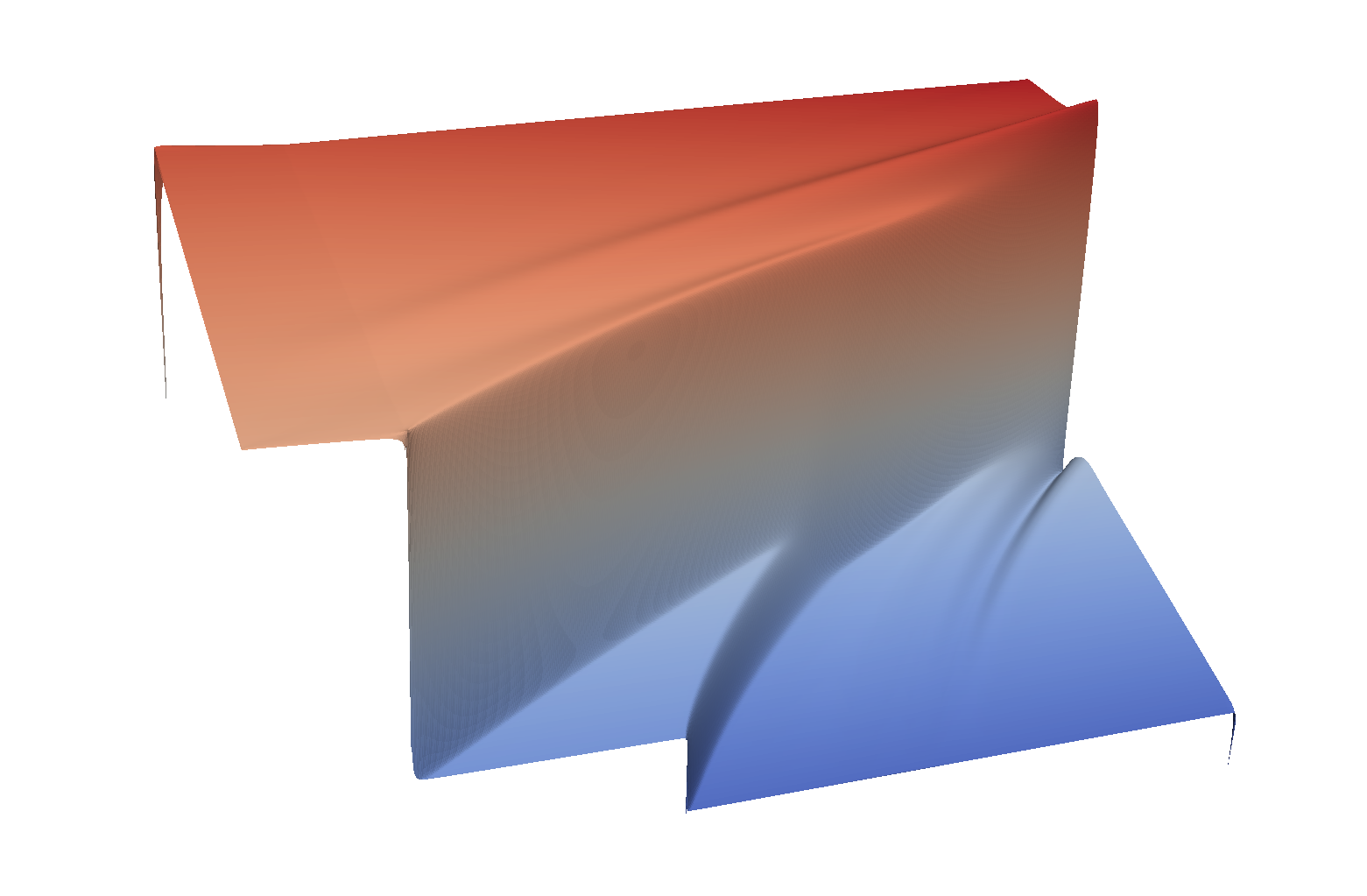}}\quad\quad
\subfloat{\includegraphics[width=0.46\textwidth]{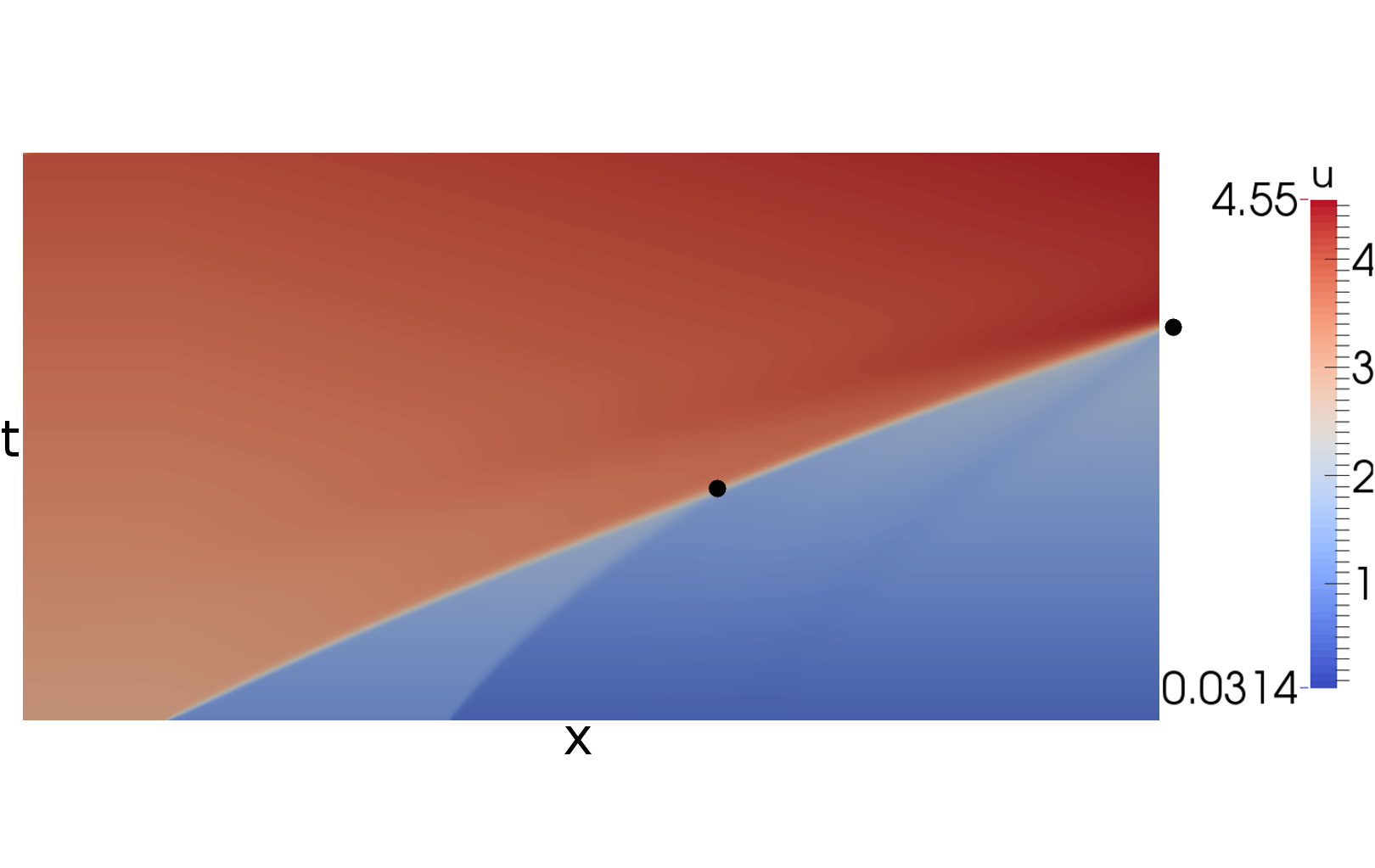}}\\
\subfloat{\includegraphics[width=0.40\textwidth]{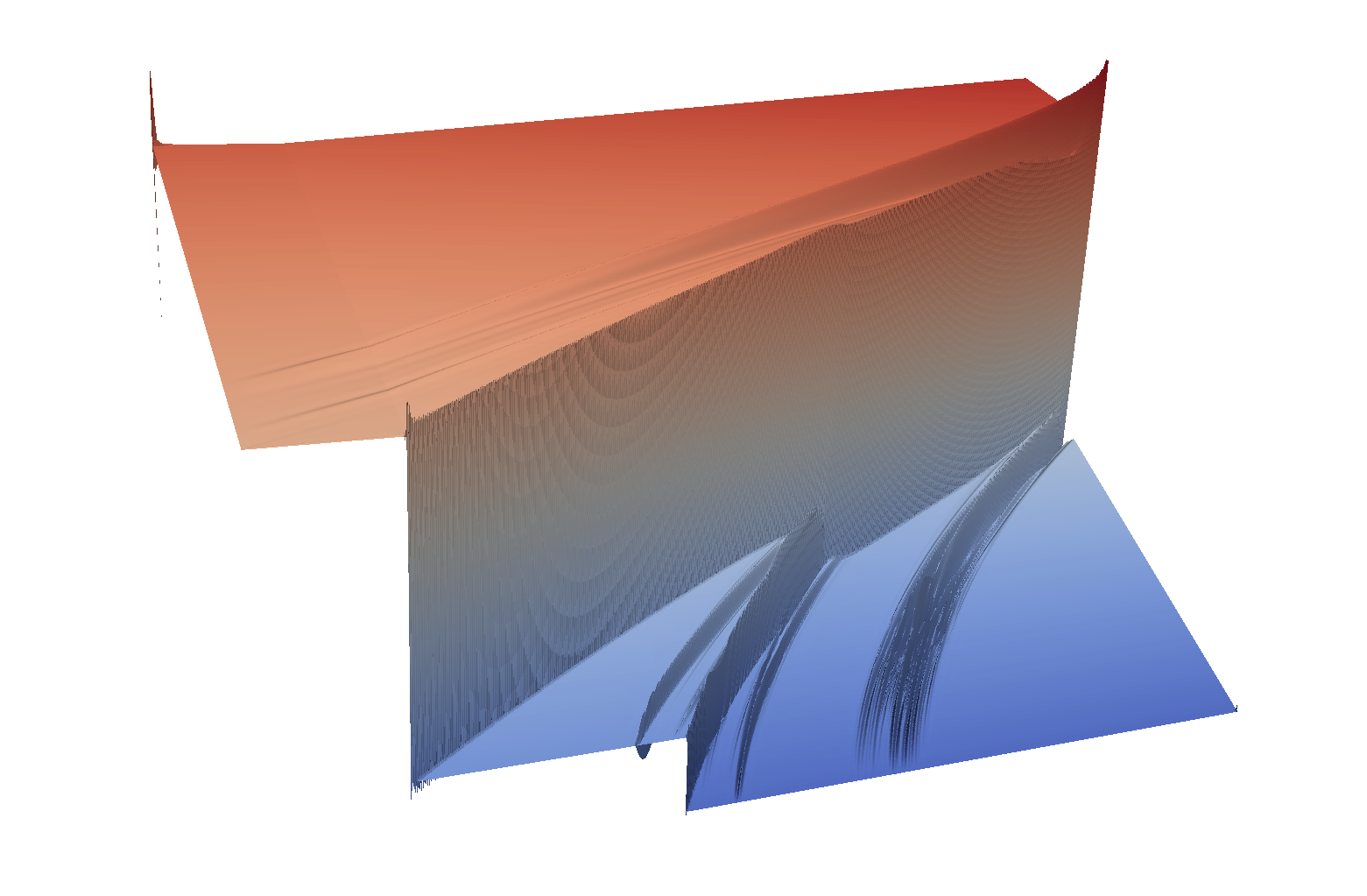}}\quad\quad
\subfloat{\includegraphics[width=0.46\textwidth]{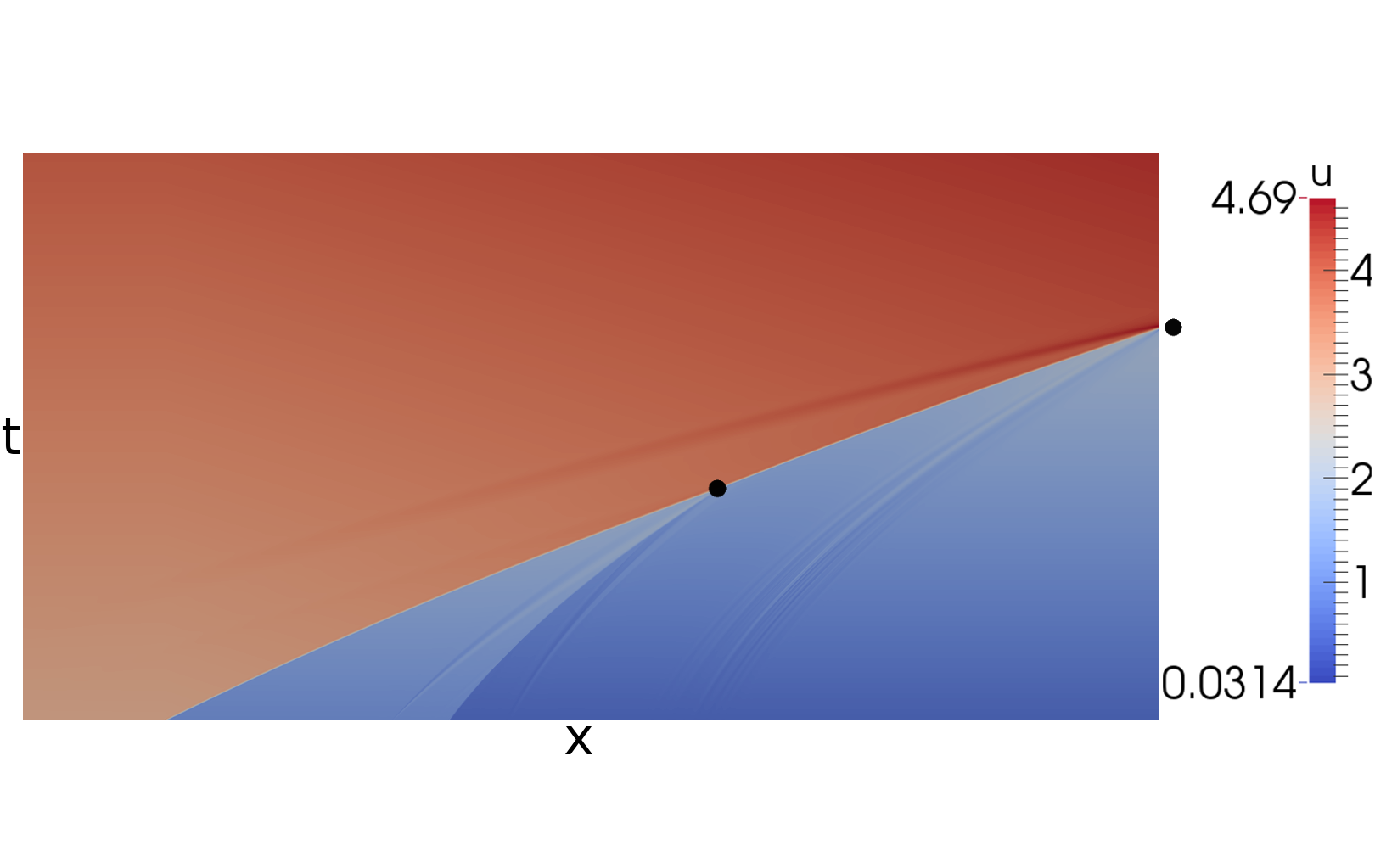}}
\caption[]{The approximation, $u^{\smallh}$, obtained from Example 3 on the finest mesh with linear $\cU^{\smallh}$; $\cVh{\Gamma_C}$,$\cVh{\Gamma_I}$ are linear on the top and quadratic on the bottom. The black dots, $\bullet$, show where the shocks collide and the resulting shock exists the domain in the exact solution, $\hu$.}\label{fig:ex3_sol}
\end{figure}

Consider \eqref{eq:balance} with
\[
r = \begin{cases} 1, &x \le 0\\ 2, &x > 0 \end{cases},\quad\quad g = \begin{cases} 3, &t = 0,\, x \le 0\\ 1, &t = 0,\, 0 < x \le 0.5\\ 0.5, &t = 0,\, x \ge 0.5\\ t + 3, &x=-0.25 \end{cases}.
\]
The respective results are shown in \cref{fig:ex3_conv,fig:ex3_sol}. Note that the method approaches the optimal convergence rate and it accurately captures the shocks and the collision point. The collision point (recall that the black dots, $\bullet$, are located at the collision and exit points) is considerably smeared when $\cVh{\Gamma_C}$,$\cVh{\Gamma_I}$ are linear, whereas this is less of an issue when they are quadratic.

The Gauss-Newton procedure utilizes a constant function as an initial guess on the coarsest mesh and, for every uniform refinement, the solution on the previous mesh is used as an initial guess. A very small tolerance is utilized, where the iteration is stopped when the change in the value of the functional, relative to the initial functional value, becomes less than $10^{-8}$. The number of Gauss-Newton iterations, for all cases and refinement levels, are shown in \cref{tbl:GNiter}. Note that the performance is expected to substantially improve by implementing adaptive mesh refinement in a nested iteration framework, which is a subject of future work; see \cref{sec:conclusions}. \blue{Since the methods are fundamentally related, the approach in this paper is expected to perform, in the context of ALR, similarly to \cite{2005HdivHyperbolic}, where with adaptive refinement the number of Gauss-Newton iterations is stable and small (at most 2 or 3) and there is, as desired, a substantially smaller increase in problem size across refinement levels in comparison to uniform refinement.} Also, in practice, a larger tolerance is sufficient and the stopping criterion can also be properly adjusted based on $h$ or the regularization, when it is utilized.

Finally, we note that using the regularized functional in \eqref{eq:reg}, with $\varepsilon = h$, recovers (both quantitatively and qualitatively) the numerical results shown in this section for the non-regularized functional. This indicates that the original functional \eqref{eq:cF}, posed on the discrete finite element spaces, possesses a corresponding implicit regularization, which can be related to the considerations in \cref{ssec:conv}.

\begin{table}
\centering
\begin{tabular}{ | l | l || r | r | r | r | r | r | }
\Xhline{2\arrayrulewidth}
Test case & \diagbox{Order}{Refs} & 0 & \hphantom{1}1 & \hphantom{1}2 & \hphantom{1}3 & 4 & 5 \\
\Xhline{4\arrayrulewidth}
\multirow{2}{*}{Example 1} & linear & 6&4&4&4&4&5 \\
\cline{2-8}
& quadratic & 8&4&4&4&5&10 \\
\Xhline{2\arrayrulewidth}
\multirow{2}{*}{Example 2} & linear & 5&3&3&3&3&3 \\
\cline{2-8}
& quadratic & 5&3&3&3&2&2 \\
\Xhline{2\arrayrulewidth}
\multirow{2}{*}{Example 3} & linear & 7&4&4&5&5&6 \\
\cline{2-8}
& quadratic & 10&5&6&8&10&5 \\
\Xcline{1-2}{2\arrayrulewidth}\Xcline{3-8}{4\arrayrulewidth}
\multicolumn{2}{c}{} & \multicolumn{6}{|c|}{\rule{0pt}{11pt}Number of iterations} \\
\Xcline{3-8}{2\arrayrulewidth}
\end{tabular}
\caption[]{\blue{Number of Gauss-Newton iterations as the mesh is refined uniformly, with a very small stopping tolerance of $10^{-8}$, for all cases and refinement levels. The space $\cU^{\smallh}$ is linear in all cases, while ``Order'' denotes the order of $\cVh{\Gamma_C}$ and $\cVh{\Gamma_I}$. ``Refs'' denotes the number of uniform refinements of the initial mesh.}}\label{tbl:GNiter}
\centering
\end{table}

\section{Conclusions and future work}
\label{sec:conclusions}

We proposed and studied a least-squares finite element formulation for hyperbolic balance laws that is based on the Helmholtz decomposition and is related to the notion of a weak solution. The ability of this approach to correctly approximate weak solutions, its convergence properties, and a special regularization were discussed; discrete convergence results were shown under mild assumptions; and numerical results were provided. The method demonstrates good convergence, shock capturing capabilities, and correctly obtains rarefaction (physically admissible) solutions to nonlinear PDEs.

There are many directions of future development. Particularly, adaptive mesh refinement in a nested iteration setting constitutes important follow-up work as it would contribute to the practical applicability of the method; extending the method to systems by utilizing a suitable Helmholtz decomposition is an important topic of future investigation; and generalizing the formulation for problems where the source term, $r$, in \eqref{eq:balancePDE} is allowed to depend linearly or nonlinearly on the unknown variable, $u$, would allow the consideration of more general hyperbolic equations. The $L^2(\Omega)$ norm convergence of the method is not a completely closed question. The regularized functional has further potential. Particularly, it can lead to more accessible proofs of convergence to the physically admissible solution and the construction of efficient linear solvers for the resulting linear systems.

\bibliographystyle{plain}
\bibliography{references}
\end{document}